\documentclass[a4paper,11pt,reqno]{amsart}

\usepackage[T1]{fontenc}

\title[Higher order Sobolev spaces between manifolds]
  {Higher order weak differentiability and Sobolev spaces between manifolds}
\author{Alexandra Convent}
\address{Alexandra Convent\\
  Universit{\'e} catholique de Louvain\\
  Institut de Recherche en Math\'ematique et Physique\\
  Chemin du Cyclotron 2 bte L7.01.01\\
  1348 Louvain-la-Neuve\\
  Belgium}
\email{Alexandra.Convent@uclouvain.be}

\author{Jean Van Schaftingen}
\address{Jean Van Schaftingen\\
  Universit{\'e} catholique de Louvain\\
  Institut de Recherche en Math\'ematique et Physique\\
  Chemin du Cyclotron 2 bte L7.01.01\\
  1348 Louvain-la-Neuve\\
  Belgium}
\email{Jean.VanSchaftingen@uclouvain.be}

\usepackage[english]{babel}
\usepackage{lmodern}
\usepackage{microtype}
\usepackage%
[top=2.7cm,bottom=2.7cm]
{geometry}
\usepackage{paralist}

\usepackage{amsmath,amssymb,amsthm,stmaryrd,accents,esint}

\usepackage{constants}

\usepackage[pdftitle={Higher order weak differentiability and Sobolev spaces between manifolds},pdfauthor={Alexandra Convent and Jean Van Schaftingen},final]{hyperref}

\usepackage[abbrev]{amsrefs}

\usepackage{tikz}
\usetikzlibrary{matrix,arrows}

\newtheorem{propoIntro}{Proposition}
\newtheorem{thm}{Theorem}[section]
\newtheorem{propo}[thm]{Proposition}
\newtheorem{lemme}[thm]{Lemma}

\theoremstyle{definition}
\newtheorem{de}{Definition}[section]
\newtheorem{example}{Example}[section]
\newtheorem{open}{Open question}[section]
\newtheorem{rem}{Remark}[section]
\DeclareMathOperator{\supp}{supp}

\DeclareMathOperator{\id}{id}

\DeclareMathOperator{\Mor}{Mor}

\DeclareMathOperator{\Vertlift}{Vert}
\providecommand{\psh}[2]{\ensuremath{\langle #1,#2\rangle}}
\providecommand{\abs}[1]{\lvert#1\rvert}
\newcommand{\bigabs}[1]{\bigl\lvert #1 \bigr\rvert}
\newcommand{\Bigabs}[1]{\Bigl\lvert #1 \Bigr\rvert}
\newcommand{\floor}[1]{\lfloor #1 \rfloor}
\newcommand{\norm}[1]{\lVert#1\rVert}
\newcommand{\bignorm}[1]{\bigl\lVert#1\bigr\rVert}
\newcommand{\double}[1]{\llbracket#1\rrbracket}

\newcommand{\R}{\mathbb{R}}
\newcommand{\N}{\mathbb{N}}
\newcommand{\Bset}{\mathbb{B}}
\newcommand{\st}{\,:\,}
\usepackage[all]{xy}
\begin{document}

\begin{abstract}
We define the notion of higher-order colocally weakly differentiable maps from a manifold $M$ to a manifold $N$. 
When $M$ and $N$ are endowed with Riemannian metrics, $p\ge 1$ and $k\ge 2$, this allows us to define the intrinsic higher-order homogeneous Sobolev space $\dot{W}^{k,p}(M,N)$. 
We show that this new intrinsic definition is not equivalent in general with the definition by an isometric embedding of $N$ in a Euclidean space; if the manifolds $M$ and $N$ are compact, the intrinsic space is a larger space than the one obtained by embedding. We show that a necessary condition for the density of smooth maps in the intrinsic space $\dot{W}^{k,p}(M,N)$ is that $\pi_{\lfloor k p \rfloor} (N) \simeq \{0\}$.
We investigate the chain rule for higher-order differentiability in this setting. 
\end{abstract}

\maketitle 

\tableofcontents
%***************************
\section{Introduction}

Higher-order Sobolev maps between manifolds are measurable maps that are at least twice weakly differentiable whose derivatives satisfy some summability condition.
As their first-order counterparts such maps are a natural framework to study \emph{geometrical objects} as biharmonic~\citelist{\cite{chang}\cite{gong}\cite{moser}\cite{struwe}\cite{urakawa}} and polyharmonic~\citelist{\cite{angels}\cite{gastel}\cite{gold}\cite{lamm}} maps and some \emph{physical models} as the Schr\"odinger--Airy and third order Landau--Lifschitz equations~\cite{sw}.

For every \(k\in \N_*\) and every \(p\in [1,+\infty)\), Sobolev maps between the manifolds \(M\) and \(N\) can be defined by first embedding the target manifold \(N\) in a Euclidean space \(\R^\nu\) for some \(\nu \in \N\) through an \emph{isometric embedding} \(\iota \in C^k (N, \R^\nu)\) 
and then considering the set~\citelist{\cite{angels}\cite{gong}\cite{moser}}
\begin{multline}\label{Wkp}
\dot{W}^{k,p}_\iota (M,N) = \bigl\{ u \colon M \to N \st \iota \circ u \in W^{k,1}_\mathrm{loc}(M, \R^\nu)
\text{ and }\abs{D^k(\iota \circ u)} \in L^p(M)\bigr\}.
\end{multline}
Since every Riemannian manifold \(N\) can be smoothly isometrically embedded in a Euclidean space~\citelist{\cite{nash56}*{theorem 3}}, this definition~\eqref{Wkp} is always possible.
In the first-order case \(k = 1\), this definition turns out to be independent of the embedding (see~\cite{cvs}*{proposition 2.7}).
However in the higher-order case \(k\ge2\), this definition \emph{is not intrinsic}: it depends on the choice of the embedding \(\iota\) (see proposition~\ref{particularIotaGeneral} below). 
The goal of this work is to propose and study an intrinsic definition of higher-order Sobolev spaces. 

In the first-order case \(k=1\), the definition by embedding~\eqref{Wkp} is also equivalent to the definition of Sobolev spaces into metric spaces~\citelist{\cite{hajlasz}*{theorem 3.2}\cite{ht}*{theorem 2.17}\cite{cvs}*{proposition 2.2}}. 
However definitions of Sobolev spaces into metric spaces \emph{do not provide any notion of weak derivative} and thus do not seem adapted to a further definition of higher-order Sobolev spaces by iteration.

For first-order Sobolev maps between manifolds, an \emph{intrinsic definition} in which weak differentiability plays a central role has been proposed in our previous work~\cite{cvs}. 
We thus follow this strategy in our aim to define and study intrinsic higher-order derivatives and Sobolev spaces.

Our definitions, results and methods apply to arbitrary order weak differentiation and Sobolev spaces, but in order to highlight the essential ideas and issues the exposition in the present introduction is focused on the second-order case.

\bigbreak

We begin by defining \emph{twice colocally weakly differentiable maps} as colocally weakly differentiable maps for which the colocal weak derivative is also colocally weakly differentiable (definition~\ref{defTwice}). 
The notion of colocal weak differentiability was introduced in~\cite{cvs}: a measurable map \(u \colon M \to N\) is \emph{colocally weakly differentiable} if for every function \(f\in C^1_c(N,\R)\), the composite function \(f\circ u : M \to \R\) is weakly differentiable~\cite{cvs}*{\S 1} and its \emph{colocal weak derivative} \(Tu \colon TM \to TN\) is a bundle morphism between the tangent spaces such that for every function \(f\in C^1_c(N,\R)\), the chain rule \(T(f\circ u) = Tf \circ\,Tu\) holds. 
The colocal weak derivative of \(Tu\) is for almost every \(x\in M\), a \emph{double vector bundle morphism} \(T^2 u(x) \colon T^2_x M \to T^2_{u(x)}N\) (proposition~\ref{propoTTU}). 

Although this second-order derivative \(T^2 u\) is a natural object from the point of view of differential geometry, \emph{it is not a vector bundle morphism} that covers \(u\) and it does not appears directly in affine or Riemannian geometry or in geometric analysis or in physical models. 
In order to remedy to this issue, we study how this object \(T^2 u\) is related to objects of affine and Riemannian geometry.

\bigbreak

When \(M\) and \(N\) are \emph{affine manifolds}, that is differentiable manifolds whose tangents manifolds are endowed with a Koszul connection which induces notion of parallel displacement and covariant derivative, the double colocal derivative is described completely by the \emph{colocal weak covariant derivative} \(D_K (Tu)\) (proposition~\ref{covariantTTU}). 

\begin{propoIntro}
\label{propoIntroHornungMoser}
Let \(K_{T^*M \otimes TN}\) be a connection on the vector bundle \(T^*M\otimes TN\). Let \(u \colon M \to N\) be a twice colocally weakly differentiable map. 
Then the map \(Tu \colon M \to T^*M\otimes TN\) has a colocal weak covariant derivative \(D_K (Tu) \colon M \to T^*M \otimes T^*M\otimes TN\) 
if and only if 
the map \(Tu \colon M \to T^*M\otimes TN\) has a colocal weak derivative \(T^2u \colon TM \to T(T^*M\otimes TN)\). 
Furthermore, they are unique and almost everywhere in \(M\) 
\begin{align*}
T^2u &= \operatorname{Vert} \circ\,D_K(Tu) + \operatorname{Hor}_K \circ\,Tu,&
&\text{and }&
D_K(Tu) &= K_{T^*M\otimes TN} \circ T^2u. 
\end{align*}
\end{propoIntro}

Roughly speaking, colocal weak covariant differentiability consists in looking only at the vertical component associated to the connection in the chain rule and relying on the formula 
\[
(Tf)(Tu)[\operatorname{Vert} \circ\,D_K Tu] = T(f \circ T u) - (Tf)(T u) [\operatorname{Hor}_K \circ\, Tu],
\]
where \(\operatorname{Vert}\) is the \emph{vertical lifting} from \(T^* M \otimes TN\) and \(\operatorname{Hor}_K\) is the \emph{horizontal lifting} associated to the connection \(K_{T^*M\otimes TN}\) (see definition~\ref{weakSecondCovariant}).

Proposition~\ref{propoIntroHornungMoser} states the equivalence between a purely differential geometric derivative and an affine geometric one; differentiability turns out to be independent of the choice of the connection.

When \(M\) and \(N\) are \emph{Riemannian manifolds} endowed with their Levi--Civita connections, the colocal weak covariant derivative is a natural concept in geometry and physics and is essentially equivalent to the weak covariant derivative of P. Hornung and R. Moser~\cite{hm}*{definition 2.5} (proposition~\ref{HornungMoserLink}).

Furthermore, this notion of covariant derivative brings us in position to define for every \(p\in [1,+\infty)\), the \emph{intrinsic second-order Sobolev space} (definition~\ref{defW2p}) 
\begin{multline*}
\dot{W}^{2,p}(M,N) 
= \bigl\{ u \colon M \to N \colon u \text{ is twice colocally weakly differentiable} \\
\text{ and } \abs{D_K(Tu)}_{g^*_M \otimes g^*_M \otimes g_N} \in L^p(M)\bigr\},
\end{multline*}
where the norm \(\abs{\cdot}_{g^*_M \otimes g^*_M \otimes g_N}\) is induced by the Riemannian metrics on \(M\) and \(N\). 
Thanks to the iterative nature of our definition of the double colocal differentiability,
we deduce immediately compactness (proposition~\ref{RellichSobolev}) and closure (proposition~\ref{closureProperty}) properties for second-order Sobolev maps.

\bigbreak 

The intrinsic second-order Sobolev spaces might by different from the one by embedding (see proposition~\ref{particularIotaGeneral} below).
As a consequence, we compare the definition of \emph{higher-order Sobolev spaces by embedding}~\eqref{Wkp} with the intrinsic one (definitions~\ref{defW2p} and~\ref{defWkp}). 
The latter definition is not intrinsic and has characterizations in terms of intrinsic spaces in several particular cases (propositions~\ref{iotaCompact}, \ref{particularIotaGeneral} and~\ref{inclusionEmbedding}). 

\begin{propoIntro}
Let \(p \in [1,+\infty)\). 
\begin{enumerate}[(i)]
\item If there exists \(C >0\) such that for every \(v \in \dot{W}^{2,p}(M,\R) \cap L^\infty(M,\R)\),
\begin{equation*}
\int_M \abs{T v}^{2p}_{g^*_M\otimes g_1} \le C \, \bigl(\operatorname{ess \, sup}\{ \abs{v(x)-v(y)} \colon x,y \in M\}\bigr)^{p} \int_M \abs{D_K(T v)}^p_{g^*_M \otimes g^*_M\otimes g_1}
\end{equation*}
and if \(N\) is compact, then for any isometric embedding \(\iota \in C^2(N,\R^\nu)\) 
\[
\dot{W}^{2,p}_\iota(M,N) = \dot{W}^{1,2p}(M,N) \cap \dot{W}^{2,p}(M,N). 
\]
\item For every isometric embedding \(\iota \in C^2(N,\R^\nu)\),
\[
\dot{W}_\iota^{2,p}(M,N) \subseteq \dot{W}^{2,p}(M,N).
\]
\item There exists an isometric embedding \(\iota \in C^2(N,\R^\nu)\) such that 
\[
\dot{W}^{2,p}_\iota(M,N) \subseteq \dot{W}^{1,2p}(M,N) \cap \dot{W}^{2,p}(M,N). 
\] 
\end{enumerate}
\end{propoIntro}
As a consequence of the latter proposition, we prove that for every \(r>0\), if \(1 \le 2p < m\) and if the manifold \(N\) is compact,
\[
\dot{W}_\iota^{2,p}(\mathbb{B}^m_r,N) \subsetneq \dot{W}^{2,p}(\mathbb{B}^m_r,N)
\]
by exhibiting a map \(u \in \dot{W}^{2,p}(\mathbb{B}^m_r,N)\) such that 
\begin{align*}
 \int_{\mathbb{B}^m_r} \abs{D_K(Tu)}^p_{g^*_m \otimes g^*_m \otimes g_N} &< + \infty &&\text{but} &&\int_{\mathbb{B}^m_r} \abs{Tu}^{2p}_{g^*_m \otimes g_N} = +\infty, 
\end{align*}
that is, \(u \notin \dot{W}^{2,p}_\iota(\mathbb{B}^m_r,N)\) (proposition~\ref{sobolevNoGN2}). 
This example implies the failure of the Gagliardo--Nirenberg interpolation inequalities~\citelist{\cite{gagliardo}\cite{nirenberg}} for those spaces. 
In the general case where the target manifold \(N\) is not compact, this leads to an open question~\ref{openQuestionGN}: does exist a constant \(C >0\) such that for every \(u\in \dot{W}^{2,p}(M,N) \cap L^\infty(M,N)\),
\[
\int_M \abs{Tu}_{g^*_M \otimes g_N}^{2p} \le C \, \bigl(\operatorname{ess \, sup} \{ d_N(u(x),u(y)) \colon x,y\in M\}\bigr)^p \int_M \abs{D_K(Tu)}_{g^*_M\otimes g^*_M \otimes g_N}^p?
\]
The answer should involve in particular the geometry of \(M\) and \(N\). 

We  consider also the question of density of the set of smooth maps \(C^k (M, N)\) in \(\dot{W}^{k, p} (M, N)\) when \(M\) and \(N\) are compact manifolds. 
In the first-order case \(k=1\), the density of smooth maps is known to be equivalent to some homotopy invariant of the pair \((M, N)\)~\citelist{\cite{hl}\cite{bethuel}}. 
Although these results are proved for embedded manifolds, the Sobolev space and the convergence are intrinsic~\cite{cvs} and the density result is thus intrinsic. 
In the higher-order case on the ball, the set of smooth maps \(C^2 (\overline{\Bset}^m, N)\) is dense in the space \(\dot{W}^{2, p}_{\iota} (\mathbb{B}^m, N)\) if and only if 
\(\pi_{\floor{2p}} \simeq \{0\}\)~\cite{bpvs}. Since the set \(\dot{W}^{2, p} (\mathbb{B}^m,N)\) is larger than \(\dot{W}^{2, p}_{\iota} (\mathbb{B}^m, N)\) but is endowed with a weaker notion of convergence (proposition~\ref{inclusionEmbedding}), it is not immediate that the topological condition is either necessary or sufficient for the density of smooth maps. 

We prove that the topological condition is still a necessary condition for the strong and weak-bounded approximation by smooth maps (proposition~\ref{convergenceHomotopy}).

\begin{propoIntro}
Let \(N\) be a smooth connected compact Riemannian manifold and let \(p\in [1,+\infty)\) so that \(1 \le 2p < m\). 
\begin{enumerate}[(i)]
\item If for every \(u \in \dot{W}^{2,p}(\mathbb{B}^m, N)\), there exists a sequence \((u_\ell)_{\ell\in \N}\) in \(C^\infty(\overline{\mathbb{B}}^m, N)\) that is bounded in \(\dot{W}^{2,p}(\mathbb{B}^m,N)\) and that converges almost everywhere, and if \(2p \notin \N\), then \(\pi_{\lfloor 2p \rfloor}(N) \simeq \{0\}\). 
\item If for every \(u \in \dot{W}^{2,p}(\mathbb{B}^m, N)\), there exists a sequence \((u_\ell)_{\ell\in \N}\) in \(C^\infty(\overline{\mathbb{B}}^m, N)\) that converges strongly to \(u\) in \(\dot{W}^{2,p}(\mathbb{B}^m,N)\), and if \(2p \in \N\), then \(\pi_{2p}(N) \simeq \{0\}\). 
\end{enumerate}
\end{propoIntro}

Giving some sufficient condition for the approximation problems seems to require new intrinsic approximation tools that would go beyond the scope of the present work. 

\bigbreak

Since Sobolev spaces by embedding~\eqref{Wkp} are defined by composition with a map \(\iota \in C^2(N,\R^\nu)\) which is an immersion, we investigate for the twice colocal weak differentiability the chain rule in view of the comparison of  the definition by embedding and the intrinsic one. 
Moreover, since the first-order colocal weak differentiability was defined by the chain rule, 
another natural candidate definition for the second-order differentiability could have been some second-order chain-rule condition. Such a definition turns out to be stronger and ill-behaved. 
Given a twice colocally weakly differentiable map \(u \colon M \to N\), we ask whether for every \(f\in C_c^2(N,\R)\), the map \(f\circ u\) is twice weakly differentiable and whether the chain rule \(T^2(f\circ u) = T^2 f \circ T^2 u\) holds. 
The chain rule \emph{does not hold in general} (example~\ref{exampleNoChainRule}); however, it holds under an additional necessary and sufficient integrability condition on colocal weak derivatives of bundle morphism (proposition~\ref{equivalentChainRule}).
\begin{propoIntro} 
Let \(u\colon M \to N\) be a twice colocally weakly differentiable map. The following statements are equivalent. 
\begin{enumerate}[(i)]
\item For every \(f\in C^2_c(N,\R)\), the map \(f\circ u\) is twice weakly differentiable and almost everywhere in \(M\) 
\[
T^2(f \circ u) = T^2 f \circ T^2 u,
\]
\item\label{conditionDoubleNorm} for every compact sets \(K \subseteq M\) and \(L \subseteq N\), 
\begin{equation*}
\int_{u^{-1}(L) \cap K} \double{T^2 u} < + \infty,
\end{equation*}
\item for every \(h \in C^1(T^*M\otimes TN,\R)\) that is linear with respect to \(Tu\) and has compact support with respect to \(u\), the map \(h \circ\,Tu\) is weakly differentiable. 
\end{enumerate}
\end{propoIntro}
The condition~\eqref{conditionDoubleNorm} involves a \emph{double norm} \(\double{\cdot}\) of the double vector bundle morphism \(T^2u\) (definition~\ref{defDoubleNorm}). 
Such a double norm is a seminorm on fibers over both vector bundle structures of the double vector bundle. 

\numberwithin{equation}{section}

%***********************************
\section{Colocal weak differentiability}
\label{sectionTwiceDifferentiable}

\subsection{Double colocal weak differentiability on differentiable manifolds}
We assume that \(M\) and \(N\) are \emph{differentiable manifolds} of class \(C^2\) of respective dimensions \(m\) and \(n\). 
As usual, the manifolds \(M\) and \(N\) satisfy the Hausdorff separation property and have a countable basis~\citelist{\cite{docarmo}*{\S 0.5}\cite{hirsch}*{\S 1.5}}.
We study the notion of second-order differentiability for measurable maps from \(M\) to \(N\).

As a preliminary we recall various definitions of local measure-theoretical notions on a manifold. A set \(E \subset M\) is \emph{negligible} if for every \(x \in M\) there exists a local chart \(\psi : V \subseteq M  \to \R^m\) -- that is \(\psi : V \subseteq M \to \psi (V) \subseteq \R^m\) is a diffeomorphism -- such that \(x \in V\) and the set \(\psi (E \cap V) \subseteq \R^m\) is negligible.
A map \(u : M \to N\) is \emph{measurable} if for every point \(x \in M\) there exists a local chart \(\psi : V \subseteq M  \to \R^m\) such that \(x \in V\) and the map \(u \circ \psi^{-1}\) is measurable~\citelist{\cite{hirsch}*{\S 3.1}\cite{derham}*{\S 3}}.
A function \(u : M \to \R\) is \emph{locally integrable} if for every \(x \in M\) there exists a local chart \(\psi : V \subseteq M \to \R^m\) such that \(x \in V\) and \(u \circ \psi^{-1}\) is integrable on \(\psi (V)\)~\cite{hormander}*{\S 6.3}.
Similarly, a locally integrable map \(u : M \to \R\) is \emph{weakly differentiable} if for every point \(x \in M\) there exists a local chart \(\psi : V \subseteq M  \to \R^m\) such that \(x \in V\) and the map \(u \circ \psi^{-1}\) is weakly differentiable. 
All these notions are independent on any particular metric or measure on the manifold \(M\) and are invariant under diffeomorphisms of \(M\).

\medbreak

Given a \emph{colocally weakly differentiable map} \(u\colon M \to N\)~\cite{cvs}*{definition 1.1},
by existence and uniqueness of the colocal weak derivative~\cite{cvs}*{proposition 1.5}, there exists a unique measurable bundle morphism \(Tu \colon TM \to TN\) such that for every \(f\in C^1_c(N,\R)\), \(T(f\circ u) = Tf \circ\,Tu\) almost everywhere in \(M\), where \((TM, \pi_M, M)\) is the \emph{tangent bundle} over \(M\), that is,
\[
TM= \bigcup_{x\in M} \{x\} \times T_x M ,
\]
\(\pi_M \colon TM \to M\) is the natural projection and for every point \(x\in M\), the fiber \(\pi_M^{-1} (\{x\})\) is isomorphic to the space \(\R^m\), 
and where a map \(\upsilon \colon TM  \to TN \) is a \emph{bundle morphism} that covers the map \(u \colon M \to N\) if the diagram
\[
 \xymatrix{
    TM \ar[r]^\upsilon \ar[d]_{\pi_M} & TN \ar[d]^{\pi_N} \\
    M \ar[r]_u & N
  }
\]
commutes, that is, we have \(\pi_N \circ \upsilon = u \circ \pi_M\) on \(TM\), and for every \(x \in M\), the map \(\upsilon(x) \colon \pi^{-1}_M(\{x\}) \simeq T_x M \to \pi^{-1}_N(\{u(x)\}) \simeq T_{u(x)} M\) is linear. 
As a consequence, the map \(Tu\) can be viewed as a measurable map from \(M\) to the bundle 
\footnote{The tensor product \((T^*M \otimes TN, \pi_{M\times N}, M \times N)\) of the bundles \(T^*M\) and \(TN\) is a vector bundle over \(M \times N\) such that for every \((x,y) \in M \times N\), \(\pi_{M\times N}^{-1}(\{(x,y)\}) = T^*_x M \otimes T_y N\).} 
\(T^*M\otimes TN\), that is, for almost every \(x\in M\), we have \(Tu(x) \in T^*_x M \otimes T_{u(x)}N = \mathcal{L}(T_x M, T_{u(x)} N)\). 
Since both \(M\) and \(N\) are manifolds of class \(C^2\), the vector bundle \(T^*M\otimes TN\) has a manifold structure of class \(C^1\) and so we can define the notion of twice colocally weakly differentiable maps recursively. 

\begin{de}\label{defTwice}
A map \(u \colon M \to N\) is \emph{twice colocally weakly differentiable} whenever \(u\) is colocally weakly differentiable and its colocal weak derivative \(T u \colon M \to T^*M\otimes TN\) is itself colocally weakly differentiable. 
\end{de}

If \(u\in W^{2,1}_\mathrm{loc}(\R^m,\R^n)\), that is, if the map \(u\colon \R^m \to \R^n \) is twice weakly differentiable, then \(u\) is twice colocally weakly differentiable. Indeed, the weak derivative \(Du \colon \R^m \to \mathcal{L}(\R^m,\R^n)\) is weakly differentiable and so for every \(f\in C^1_c(\mathcal{L}(\R^m, \R^n),\R)\), the composite map \(f\circ D u\) is weakly differentiable. 
But, as in the case of colocally weakly differentiable maps, \emph{the converse is false}. 
For example, the function \(u\colon \R^m \to \R\) defined for every \(x \in \R^m \setminus \{0\}\) by \(u(x) = \abs{x}^{-\alpha}\) does not belong to \(W^{2,1}_\mathrm{loc}(\R^m)\) for any \(\alpha > m - 2\), but is twice colocally weakly differentiable for every \(\alpha \in \R\). 
The boundedness of \(u\) is not essential: indeed for every \(\alpha < m - 1\), the function \((\cos \circ\, u, \sin \circ\, u) : \R^m \to \mathbb{S}^1 \subseteq \R^2\) is twice colocally weakly differentiable but it does not belong to \(W^{2,1}_\mathrm{loc}(\R^m,\R^2)\) if \(\alpha > \frac{m - 2}{2}\). 

The natural framework of the colocal weak derivative of a bundle morphism involves the \emph{second-order tangent bundle} \(T^2 M := T(TM)\)~\citelist{\cite{bertram}*{\S 7.1}\cite{michor}*{\S 8.13}}; the natural projection \(\pi^2_M \colon T^2M \to M\) \emph{is not a vector bundle} but the tangent manifold \(T^2 M\) has two natural vector bundle structures as a bundle over \(TM\)~\citelist{\cite{bertram}*{\S 9.1}\cite{dieudonne}*{\S 16.15.7}\cite{mackenzie}*{example 1.3}\cite{michor}*{\S 8.12}}, that is, 
\begin{compactitem}[--]
\item the canonical vector bundle structure of the tangent bundle \(\pi_{TM} \colon T^2 M \to TM\),
\item the tangent vector bundle structure on the tangent fibration \(T\pi_M \colon T(TM) \to TM\);
\end{compactitem}
hence, \(T^2 M\) is a \emph{double vector bundle}~\citelist{\cite{gr}*{theorem 3.1}\cite{mackenzie}*{definition 1.1}}, that is, a system of four vector bundle structures such that the following diagram commutes 
\begin{equation*}
 \xymatrix{
    T^2 M \ar[r]^{\pi_{TM}} \ar[d]_{T\pi_M} & TM \ar[d]^{\pi_M } \\
    TM \ar[r]_{\pi_M} & M
  }
\end{equation*}
and such that each of the four structure maps (namely, the bundle projection, the zero section, addition and scalar multiplication) of each vector bundle structure on \(T^2 M\) is a morphism of vector bundles with respect to the other structure; 
the double vector bundle \(T^2M\) has one more natural structure, the \emph{canonical flip} \(\kappa_M \colon T^2M \to T^2M\) which is a linear isomorphism from the bundle \((T^2M, T\pi_M, TM)\) to the bundle \((T^2M, \pi_{TM}, TM)\)~\cite{michor}*{\S 8.13}; 
the space of all \emph{double vector bundle morphisms} \(\mathrm{Mor}(T^2 M, T^2 N)\) is defined by requiring that for every \((x,y)\in M \times N\), the fiber \(\mathrm{Mor}(T^2 M, T^2 N)_{x,y}\) consists of maps \(\upsilon \colon T^2_x M \to T^2_y N\) and \(\upsilon_1, \upsilon_2 \in \mathcal{L}(T_x M, T_y N)\) such that \(\pi_{TN} \circ \upsilon = \upsilon_1 \circ \pi_{TM}\) and \(T\pi_N \circ \upsilon = \upsilon_2 \circ T\pi_M\); 
we note that \(\mathrm{Mor}(T^2M, T^2N)\) is a double vector bundle with the following commutative diagram
\begin{center}
\begin{tikzpicture}[back line/.style={densely dotted},scale=3]
\node (P0) at (0,0) {$\Mor(T^2M, T^2N)$};
\node (P1) at (0,-0.5) {$T^*M \otimes TN $} ;
\node (P2) at (1.5,0) {$T^*M \otimes TN $};
\node (P3) at (1.5,-0.5) {$M \times N$};

\draw 
(P0) edge[->] node[above] {\tiny{$\pi^1_{T^*M\otimes TN}$}} (P2)
(P0) edge[->] node[left] {\tiny{$\pi^2_{T^*M\otimes TN}$}} (P1)

(P1) edge[->] node[below] {\tiny{$\pi_{M\times N}$}} (P3)
(P2) edge[->] node[right] {\tiny{$\pi_{M\times N}$}} (P3);
\end{tikzpicture}
\end{center}
where for every \(\upsilon \in \mathrm{Mor}(T^2M, T^2N)\), \(\pi^1_{T^*M\otimes TN} \circ \upsilon = \pi_{TN} \circ \upsilon \circ 0_{\pi_{TM}}\) and \(\pi^2_{T^*M\otimes TN} \circ \upsilon = T\pi_{N} \circ \upsilon \circ 0_{T\pi_M}\), with the two zero sections \(0_{\pi_{TM}} \colon TM \to T^2M\) and \(0_{T\pi_M} \colon TM \to T^2M\). 
If \(u \in C^2 (M, N)\), then the second-order tangent map \(T^2 u \colon T^2M \to T^2N\) and the maps \(Tu\colon TM\to TN\) and \(u \colon M \to N\) define a \emph{morphism of double vector bundles}~\citelist{\cite{mackenzie}*{definition 1.2}} such that the following diagrams commute.
\begin{center}
\begin{tikzpicture}[back line/.style={densely dotted},scale=3]
\node (P0) at (0,0,0) {$T^2 M$};
\node (P1) at (0,0,1) {$TM$} ;
\node (P2) at (0,-1,1) {$M$};
\node (P3) at (1,0,0) {$T^2 N$};
\node (P4) at (1,0,1) {$TN$};
\node (P5) at (1,-1,1) {$N$};
\node (P6) at (0,-1,0) {$TM$};
\node (P7) at (1,-1,0) {$TN$};

\draw 
(P0) edge[->] node[above] {\tiny{$T^2u$}} (P3)

(P0) edge[->] node[near end, left] {\tiny{$T\pi_M$}} (P6)
(P0) edge[->] node[left] {\tiny{$\pi_{TM}$}}(P1)

(P3) edge[->] node[right] {\tiny{$T\pi_N$}} (P7)
(P3) edge[->] node[left] {\tiny{$\pi_{TN}$}} (P4)

(P1) edge[->] node[near end, above] {\tiny{$Tu$}} (P4)
(P6) edge[->] node[near start, above] {\tiny{$Tu$}} (P7)

(P6) edge[->] node[at start, left] {\tiny{$\pi_M$}} (P2)
(P1) edge[->] node[left] {\tiny{$\pi_M$}} (P2)

(P4) edge[->] node[near start, right] {\tiny{$\pi_N$}} (P5)
(P7) edge[->] node[right] {\tiny{$\pi_N$}} (P5)

(P2) edge[->] node[above] {\tiny{$u$}} (P5);
\end{tikzpicture}
\begin{tikzpicture}[back line/.style={densely dotted},scale=3]
\node (P0) at (0,0,0) {$T^2 M$};
\node (P1) at (0,0,1) {$T^2 M$} ;
\node (P2) at (0,-1,1) {$TM$};
\node (P3) at (1,0,0) {$T^2 N$};
\node (P4) at (1,0,1) {$T^2 N$};
\node (P5) at (1,-1,1) {$TN$};
\node (P6) at (0,-1,0) {$TM$};
\node (P7) at (1,-1,0) {$TN$};
\draw 
(P0) edge[->] node[above] {\tiny{$T^2 u$}} (P3)

(P0) edge[->] node[near end, left] {\tiny{$T\pi_M$}} (P6)
(P0) edge[->,bend right] node[left] {\tiny{$\kappa_M$}}(P1)

(P1) edge[->,bend right] node[left] {\tiny{$\kappa_M$}}(P0)
(P1) edge[->] node[near end, above] {\tiny{$T^2 u$}} (P4)
(P1) edge[->] node[left] {\tiny{$\pi_{TM}$}} (P2)

(P2) edge[->] node[above] {\tiny{$Tu$}} (P5)

(P3) edge[->] node[right] {\tiny{$T\pi_N$}} (P7)
(P3) edge[->,bend right] node[left] {\tiny{$\kappa_N$}} (P4)
(P4) edge[->,bend right] node[left] {\tiny{$\kappa_N$}} (P3)
(P4) edge[->] node[near start, right] {\tiny{$\pi_{TN}$}} (P5)

(P6) edge[->] node[near start, above] {\tiny{$Tu$}} (P7)

(P6) edge[double equal sign distance] (P2) 
(P7) edge[double equal sign distance] (P5);
\end{tikzpicture}
\end{center}
Moreover, the map \(T^2u\) can be viewed as a section from  \(M\) to \(\Mor(T^2M, u^*(T^2M))\), that is, for every \(x\in M\), \(T^2u(x) \in \Mor(T^2M,T^2N)_{x,u(x)}\). 
However, twice colocally weakly differentiable maps are just measurable and so we say that a measurable map \(v\colon M \to \Mor(T^2M,T^2N)\) covers \(u\colon M \to N\) whenever for almost every \(x \in M\), 
\[
\pi_{M\times N} \circ \pi^1_{T^*M\times TN} \circ v(x) = \pi_{M\times N} \circ \pi^2_{T^*M\times TN} \circ v(x) = \bigl(x,u(x)\bigr).
\]

If \((E,\pi_M,M)\) is a vector bundle, the \emph{vertical bundle} \(VE \to E\) is the subbundle of \(TE \to E\) defined by~\cite{wendl}*{definition 3.1}
\[
VE = \bigl\{ \nu \in TE \colon T\pi_M(\nu) = 0\bigr\} = \operatorname{ker}  T\pi_M,
\]
and for every \(e\in E\), the \emph{vertical lift} \(\operatorname{Vert}_e \colon E \to TE\) defined 
for \(\nu \in \pi^{-1}_M(\{\pi_M(e)\})\) by 
\[
\operatorname{Vert}_e (\nu) = \frac{d}{dt}  (e+t \nu) \Big\vert_{t = 0}
\]
gives a natural isomorphism between \(E_{\pi_M(e)}\) and \(V_e E\). 
This construction is independent of any connection. 

\begin{propo}\label{propoTTU}
If the map \(u\colon M \to N\) is twice colocally weakly differentiable, then there exists a unique measurable map \(T^2 u\colon M \to \Mor(T^2 M, T^2N)\) that covers the map \(u\) such that for every function \(f\in C^1 (T^*M\otimes TN,\R)\) such that \(f \circ T u \in W^{1, 1}_{\mathrm{loc}} (M)\),
\[
T(f \circ T u) = T f \circ T^2 u
\]
almost everywhere in \(M\). 
Moreover, \(\kappa_N \circ \, T^2u = T^2u \circ \, \kappa_M\) and \(T^2 u \circ \operatorname{Vert}_{TM} = \operatorname{Vert}_{TN} \circ\, Tu\) almost everywhere in \(M\).
\end{propo}

In other words, the map \(T^2 u\) is the colocal weak derivative of the bundle morphism \(Tu\). 
Moreover, for almost every \(x\in M\), the double vector bundle morphism \(T^2u(x) \in\) \(\Mor(T^2M, T^2N)_{x,u(x)}\) satisfies \(\pi^1_{T^*M \otimes TN} \circ T^2u(x) = \pi^2_{T^*M\otimes TN} \circ T^2u(x) = Tu(x)\) and 
\(T^2 u(x) \circ \operatorname{Vert}_{TM} = \operatorname{Vert}_{TN} \circ\,Tu(x)\). 

The first lemma recalls how these properties are obtained when the target manifold is a Euclidean space. 
We remind the reader that a locally integrable map \(u\colon M \to \R\) is \emph{twice weakly differentiable} if for every \(x\in M\), there exists a local chart \(\psi \colon V \subseteq M \to \R^m\) such that \(x\in V\) and the map \(u \circ \psi^{-1}\) belongs to \(W^{2,1}_\mathrm{loc}(\psi(V))\).

\begin{lemme}\label{lemmeTTU}
If the locally integrable map \(v\colon M \to \R^n\) is twice weakly differentiable, then there exists a unique measurable map \(T^2 v\colon M \to \Mor(T^2 M, T^2\R^n)\) that covers the map \(v\) such that 
for every local chart \(\psi \colon V \subseteq M \to \R^m\), 
\[
T^2(v \circ \psi^{-1}) = T^2 v \circ T^2 (\psi^{-1})
\]
almost everywhere in \(\psi(V)\). 
Moreover, \(T\pi_{\R^n} \circ T^2v = Tv \circ T\pi_M\), \(\kappa_{\R^n} \circ T^2v = T^2v \circ \kappa_M\) and \(T^2 v \circ \operatorname{Vert}_{TM} = \operatorname{Vert}_{T\R^n} \circ\, Tv\) almost everywhere in \(M\). 
\end{lemme}

\begin{proof}[Sketch of the proof]
First, if \(M\) is an open set of \(\R^m\), since the domain is flat, for almost every \(x\in M\), \(T^2v(x) \colon T^2_x M \to T_{v(x)}^2\R^n\) is defined for every \((e,\mu,\nu) \in T^2_x M\) by
\[
T^2v(x)[e,\mu,\nu] = (Dv(x)[e], Dv(x)[\mu], D^2v(x)[e,\mu] + Dv(x)[\nu]), 
\]
and the conclusion follows. 

Next, for a general manifold \(M\), for every local chart \(\psi \colon V \subseteq M \to \R^m\), the map \(v\circ \psi^{-1} \colon \psi(V) \subseteq \R^m \to \R^n\) is twice weakly differentiable. 
Since \(\psi \in C^2(V,\R^m)\), we have \(T^2\psi \in \Mor(T^2 V, \psi^*(T^2\R^m))\), \(\kappa_{\R^m} \circ T^2\psi = T^2 \psi \circ \kappa_M\) and \(T^2 \psi \circ \operatorname{Vert}_{TM} = \operatorname{Vert}_{T\R^m} \circ \, T\psi\)~\cite{michor}*{\S 8.13}. 
As a consequence, since almost everywhere in \(V\)
\[
T^2 v = T^2(v \circ \psi^{-1}) \circ T^2\psi, 
\]
by the previous step, we can deduce that \(T\pi_{\R^n} \circ T^2 v = T v \circ T\pi_M\), \(\kappa_{\R^n} \circ T^2v = T^2v \circ \kappa_M\) and \(T^2 v \circ \operatorname{Vert}_{TM} = \operatorname{Vert}_{T\R^n} \circ\, Tv\) almost everywhere in \(V\). 
Since \(\pi_{TM} = T\pi_M \circ \kappa_M\) on \(T^2M\) and \(\pi_{T\R^n} = T\pi_{\R^n} \circ \kappa_{\R^n}\) on \(T^2\R^n\)~\cite{michor}*{\S 8.13}, then \(\pi_{T\R^n} \circ T^2v = Tv \circ \pi_{TM}\) almost everywhere in \(V\) 
and so for almost every \(x\in V\), \(T^2v(x) \in \Mor(T^2M, T^2N)_{x,v(x)}\). 
Finally, since the manifold \(M\) has a countable atlas of local charts, by a direct covering argument, the previous equalities are satisfied almost everywhere in \(M\). 
\end{proof}

The following lemma~\ref{lemmeExtended} allows to embed a local chart into a compactly supported function; its proof and its statement is similar to the one for \(C^1\)--manifolds~\cite{cvs}*{lemma 1.6}.
\begin{lemme}[Extended local charts]\label{lemmeExtended}
For every \(y \in N\), there exist an open subset \(U\subseteq N\) such that \(y\in U\), and maps \(\varphi \in C^2(N,\R^n)\) and \(\varphi^* \in C^2(\R^n,N)\) such that 
\begin{enumerate}[(i)]
  \item the set \(\supp \varphi\) is compact,
  \item the set \(\overline{\{ x \in \R^n \colon \varphi^*(x) \neq y\}}\) is compact,
  \item the map \(\varphi_{\arrowvert U}\) is a diffeomorphism of class \(C^2\) onto its image \(\varphi(U)\),
  \item \(\varphi^* \circ \varphi= \id\) in \(U\).
\end{enumerate}
\end{lemme}

\begin{proof}
By definition of manifold of class \(C^2\), there exists a local chart \(\psi \colon V \subseteq N \to \psi(V) \subseteq \R^n\) around \(y \in N\) such that \(\psi \colon V \to \psi(V)\) is a diffeomorphism of class \(C^2\). 
Without loss of generality, we assume that \(\psi(y) = 0\). 
Since the set \(\psi (V) \subseteq \R^n\) is open, there exists \(r>0\) such that \(B_{2 r} \subseteq \psi(V)\). We choose a function \(\theta \in C^2_c(\R^n,\R)\) such that \(0\le \theta \le 1\) on \(\R^n\), \(\theta = 1\) on \(B_r\) and \(\supp \theta \subset B_{2 r}\). 
We take the set \(U = \psi^{-1}(B_r)\) and the maps \(\varphi \colon N \to \R^n\) defined for every \(z \in N\) by
\begin{equation*}
\varphi(z)=
\begin{cases}
\theta\bigl(\psi(z)\bigr)\, \psi(z) & \text{ if } z \in \psi^{-1}(B_{2 r}), \\
0 & \text{ otherwise }\\
\end{cases}
\end{equation*}
and \(\varphi^*\colon \R^n \to N\) defined for every \(x\in \R^n\) by \(\varphi^*(x) = \psi^{-1}(\theta(x) \,x).\)
\end{proof}

\begin{proof}[Proof of proposition~\ref{propoTTU}]
By the existence and uniqueness of the colocal weak derivative~\cite{cvs}*{proposition 1.5}, there exists a unique measurable bundle morphism \(T^2 u \colon TM \to T(T^*M \otimes TN)\) such that for every \(f\in C^1 (T^*M\otimes TN,\R)\) such that \(f \circ\,Tu \in W^{1, 1}_{\mathrm{loc}} (M)\), we have \(T(f\circ\,Tu) = Tf \circ T^2 u\) almost everywhere in \(M\). 
As for the colocal weak derivative, the map \(T^2u\) can be viewed as a measurable map from \(M\) to \(T^*M \otimes T(T^*M\otimes TN)\), which is a subset of \(\Mor(T^2M,T^2N)\). 

Let \(U\subseteq N\), \(\varphi \in C^2(N,\R^n)\) and \(\varphi^* \in C^2(N,\R^n)\) be the extended local charts given by lemma~\ref{lemmeExtended}. 
For every \(y \in T^*M\otimes TU\), there exist two open sets \(V,W \subseteq T^*M\otimes TU\) such that \(y \in \overline{V} \subseteq W\) and then a map \(\eta \in C^1_c(T^*M\otimes TN,\R)\) such that \(0 \le \eta \le 1\) on \(T^*M\otimes TN\), \(\eta = 1\) on \(V\) and \(\supp(\eta) \subset W\).
As a consequence, the function \(h\) defined for every \(z\in T^*M\otimes TN\) by 
\begin{equation*}
h(z)=
\begin{cases}
\eta(z) T\varphi(z) & \text{ if } z \in W, \\
0 & \text{ otherwise, }\\
\end{cases}
\end{equation*}
belongs to \(C^1_c(T^*M\otimes TN, \R^n)\) and satisfies \(T(h \circ\,Tu) = Th \circ T^2 u\) almost everywhere in \(M\). 

Since \(\eta =1\) on \(V\), we have in particular \(T^2(\varphi \circ u) = T^2\varphi \circ T^2u\) almost everywhere in \(u^{-1}(\pi_N(\pi_{M\times N}(V)))\).  
Since \(T^2\varphi^* \circ T^2\varphi = \mathrm{id}\) on \((\pi^2_N)^{-1}(U)\), by lemma~\ref{lemmeTTU} and since \(\varphi^*\in C^2(\R^n,N)\)~\cite{michor}*{\S 8.13}, almost everywhere in \(u^{-1}(\pi_N(\pi_{M\times N}(V)))\) 
\begin{align*}
\kappa_N \circ T^2 u 
& = \kappa_N \circ T^2 \varphi^* \circ T^2(\varphi \circ u) \\
& = T^2\varphi^* \circ \kappa_{\R^n} \circ T^2(\varphi \circ u) \\
& = T^2\varphi^* \circ T^2(\varphi \circ u) \circ \kappa_M = T^2 u \circ \kappa_M. 
\end{align*}
In particular, we have 
\[
 T^2 u = \kappa_N \circ T^2u \circ \kappa_M,
\]
which gives a second bundle morphism structure to \(T^2u\) so that \(T^2 u\) is a double vector bundle morphism in \(u^{-1}(\pi_N(\pi_{M\times N}(V)))\).

The facts that \(T^2 u \circ \operatorname{Vert}_{TM} = \operatorname{Vert}_{TN} \circ\,Tu\) and \(T\pi_N \circ T^2 u = Tu \circ T\pi_M\) almost everywhere in \(u^{-1}(\pi_N(\pi_{M\times N}(V)))\) can be proved along the same line. 
Since \(\pi_{TM} = T\pi_M \circ \kappa_M\) on \(T^2M\) and \(\pi_{TN} = T\pi_N \circ \kappa_N\) on \(T^2N\)~\cite{michor}*{\S 8.13}, 
\begin{align*}
Tu \circ \pi_{TM} 
= T u \circ T\pi_{M} \circ \kappa_M = T\pi_N \circ T^2 u\circ \kappa_M 
= T\pi_N \circ \kappa_N \circ T^2 u = \pi_{TN} \circ T^2 u
\end{align*}
almost everywhere in \(u^{-1}(\pi_N(\pi_{M\times N}(V)))\). 
So for almost every \(x\in u^{-1}(\pi_N(\pi_{M\times N}(V)))\), 
\[
\pi^1_{T^*M \otimes TN} \circ T^2 u(x) = \pi^2_{T^*M\otimes TN} \circ T^2 u(x) = Tu(x)
\]
and \(\pi_{M \times N} \circ\,Tu(x) = (x,u(x))\). 

Finally, by a direct covering argument, since the manifolds \(M\) and \(N\) have a countable basis, the measurable map \(T^2u \colon M \to \Mor(T^2M,T^2N)\) satisfies the previous properties almost everywhere in \(M\). 
\end{proof}

%**************************************
\subsection{Sequences of twice colocally weakly differentiable maps}
In order to study the lower semi-continuity of functionals, it is interesting to have some sufficient conditions in the calculus for a limit of twice colocally weakly differentiable maps to be also twice colocally weakly differentiable. 
In this part, we deduce from the closure property for first-order derivatives~\cite{cvs}*{proposition 3.7} a closure property for second-order derivatives. 

We first recall a notion of convergence in measure~\cite{cvs}*{definition 3.5}. 
\begin{de}\label{defConvergeMeasure}
A sequence \((u_\ell)_{\ell\in N}\) of maps from \(M\) to \(N\) \emph{converges locally in measure} to a map \(u\colon M \to N\) whenever for every \(x\in M\) there exists a local chart \(\psi \colon V \subseteq M \to \R^m\) such that \(x\in V\) and for every open set \(U \subseteq N\), 
\[
\lim_{\ell \to \infty} \mathcal{L}^m\bigl(\psi((u^{-1}(U) \cap V) \setminus u^{-1}_\ell(U))\bigr) = 0.
\]
\end{de}
The definition applies directly to a sequence \((\upsilon_\ell)_{\ell \in \N}\) of bundle morphisms between \(TM\) and \(TN\) viewed as maps from \(M\) to \(T^*M \otimes TN\). 

We then generalize the notion of uniform integrability~\cite{cvs}*{definition 3.6} to any bundle morphisms. 

\begin{de}\label{defUniformIntegrable}
Let \((F,\pi_N,N)\) be a vector bundle. 
A sequence \((\upsilon_\ell)_{\ell \in \N}\) of bundle morphisms from \(TM\) to \(F\) that covers a sequence \((u_\ell)_{\ell \in \N}\) of maps from \(M\) to \(N\) is \emph{bilocally uniformly integrable} whenever for every \(x\in M\) and every \(y\in N\), there exist a local chart \(\psi \colon V \subseteq M \to \R^m\) and a bundle chart \(\varphi \colon \pi_N^{-1}(U) \subseteq F \to \R^q\) (\(\varphi \colon \pi_N^{-1}(U) \subseteq F \to \varphi(\pi_N^{-1}(U)) \subseteq \R^q\) is a diffeomorphism and \(U\subseteq N\) is open) such that for every \(\varepsilon >0\) there exists \(\delta >0\) such that if \(W \subseteq \psi(V)\), \(\mathcal{L}^m(W)\le \delta\) and \(\ell \in \N\), 
\[
\int_{W \cap \psi(u_\ell^{-1}(U))} \abs{\varphi \circ \upsilon_\ell \circ T(\psi^{-1})}\le \varepsilon. 
\]
\end{de}
This definition applies directly to a sequence \((\upsilon_\ell)_{\ell \in \N}\) from \(TM\) to \(TN\) or from \(TM\) to \(T(T^*M \otimes TN)\). 

\begin{propo}\label{closurePropertyTwice}
Let \((u_\ell)_{\ell \in \N}\) be a sequence of twice colocally weakly differentiable maps from \(M\) to \(N\). 
If the sequence \((Tu_\ell)_{\ell \in \N}\) converges locally in measure to a measurable map \(\upsilon \colon M \to T^*M \otimes TN\) that covers a map \(u\colon M \to N\), if the sequences \((Tu_\ell)_{\ell \in \N} \) and \((T^2 u_\ell)_{\ell \in \N}\) are bilocally uniformly integrable, then the map \(u\) is twice colocally weakly differentiable. 
\end{propo}

This proposition is a direct consequence of the closure property for colocally weakly differentiable maps~\cite{cvs}*{proposition 3.7}, applied to the sequence \((u_\ell)_{\ell \in \N}\) and then to \((Tu_\ell)_{\ell\in \N}\) . 

%**********************************
\subsection{Higher order colocal weak differentiability}

In this part, we explain how previous notions and propositions extend to higher-order derivatives. 

For \(k \ge 3\), we assume that \(M\) and \(N\) are manifolds of class \(C^k\). 
Since the set \(T^*M \otimes TN\) is a manifold of class \(C^{k-1}\), we also proceed by induction. 
A map \(u\colon M \to N\) is \(1\) times colocally weakly differentiable whenever it is colocally weakly differentiable. 

\begin{de}
Let \(k \in \N_*\). A map \(u \colon M \to N\) is \emph{\(k\) times colocally weakly differentiable} whenever \(u\) is \((k-1)\) times colocally weakly differentiable and \(T u \colon M \to T^*M \otimes TN\) is \((k-1)\) times colocally weakly differentiable. 
\end{de}

As for the notion of twice weak differentiability, there is some natural framework to work with, that is, we denote by \(T^k M\) the \emph{\(k^{\text{th}}\) order tangent bundle} which is by induction \(T^k M := T(T^{k-1} M)\)~\cite{bertram}*{\S 7.1}; 
this manifold has \(k\) natural vector bundle structures \(\pi_{T^{k-1}M}^1,\dots, \pi_{T^{k-1}M}^k\) over \(T^{k-1}M\)~\cite{bertram}*{\S 15.1} and it is a \emph{\(k^\text{tuple}\) vector bundle}~\citelist{\cite{gm}*{definition 3.1}\cite{gr}*{definition 4.1 and theorem 5.1}}; the canonical flips correspond to any permutation of the bundle structures and 
the space of \emph{\(k^\text{tuple}\) vector bundle morphisms} \(\mathrm{Mor}(T^k M, T^k N)\) is  defined by the condition that for every \((x,y)\in M \times N\), the fiber \(\mathrm{Mor}(T^k M, T^k N)_{x,y}\) consists of maps \(\upsilon \colon T^k_x M \to T^k_y N\) and \(\upsilon_1, \dotsc, \upsilon_k \in \Mor(T^{k-1} M, T^{k-1} N)_{x,y}\) such that for every \(j \in \{1,\dots,k\}\), \(\pi_{T^{k-1}N}^j \circ \upsilon = \upsilon_j \circ \pi_{T^{k-1}M}^j\)~\citelist{\cite{gm}*{definition 3.1}\cite{gr}*{definition 4.1}};
we note that \(\mathrm{Mor}(T^k M, T^kN)\) has also a \(k^\text{tuple}\) vector bundle structure. 

If \(u\colon M \to N\) is a \(k\) times colocally weakly differentiable map, then by existence and uniqueness of the colocal weak derivative~\cite{cvs}*{proposition 1.5}, and by induction, there exist unique measurable maps \(T^2 u,\dots, T^k u\) such that for every \(j \in \{3,\dots,k\}\), \(T^j u \colon TM \to T(\Mor(T^{j-1} M,\) \(T^{j-1} N))\) is the colocal weak derivative of \(T^{j-1}u\); this derivative can also be viewed as a map from \(M\) to \(\Mor(T^j M,T^j N)\) that covers \(u\), that is, for almost every \(x\in M\), \(T^j u(x) \in \Mor(T^j M, T^jN)_{x,u(x)}\). 

\medbreak 

We have a closure property for higher-order colocally weakly differentiable maps similar to proposition~\ref{closurePropertyTwice} for twice colocally weakly differentiable maps.  
Indeed, definition~\ref{defConvergeMeasure} applies directly to a sequence \((\upsilon_\ell)_{\ell \in \N}\) of maps from \(M\) to \(\Mor(T^k M, T^k N)\). 
Moreover, definition~\ref{defUniformIntegrable} applies also to a sequence \((\upsilon_\ell)_{\ell \in \N}\) from \(TM\) to \(T(\Mor(T^{k-1}M, T^{k-1} N))\) that covers a sequence \((u_\ell)_{\ell \in \N}\) from \(M\) to \(\Mor(T^{k-1}M, T^{k-1} N)\). 

\begin{propo}\label{closurePropertyKTimes}
Let \((u_\ell)_{\ell \in \N}\) be a sequence of \(k\) times colocally weakly differentiable maps from \(M\) to \(N\).
If the sequence \((T^{k-1} u_\ell)_{\ell \in \N}\) converges locally in measure to a measurable map \(\upsilon \colon M \to \Mor(T^{k-1} M, T^{k-1} N)\) that covers a map \(u\colon M \to N\) and if for all \(j\in \{1,\dots, k\}\), the sequence \((T^j u_\ell)_{\ell \in \N} \) is bilocally uniformly integrable, then the map \(u\) is \(k\) times colocally weakly differentiable. 
\end{propo}

Propositon~\ref{closurePropertyKTimes} is a direct consequence of the closure property for colocally weakly differentiable maps~\cite{cvs}*{proposition 3.7}, applied recursively to the sequences \((u_\ell)_{\ell \in \N},\) \(\dots, (T^{k-1} u_\ell)_{\ell\in \N}\). 

%*************************************
\section{Colocal weak covariant derivatives and Sobolev spaces}\label{sectionCovariant}

In this section we study how differentiability can be characterized in covariant terms when the manifolds \(M\) and \(N\) have an affine structure.

%************************************
\subsection{Geometric preliminaries}\label{sectionPreliminaries}
We recall some concepts and tools of affine geometry of a vector bundle \((E,\pi_M,M)\). 
A map \(K_E \colon TE \to E\) is a \emph{connection} if \(K_E\) is a bundle morphism from \((TE, \pi_E, E)\) to \((E, \pi_M, M)\) that covers \(\pi_M\) and a bundle morphism from \((TE, T \pi_M, TM)\) to \((E, \pi_M, M)\) that covers \(\pi^{TM}_{M} \colon TM \to M\), and for every \(e \in E\), \((K_E \circ \Vertlift)(e) = e\)~\cite{wendl}*{definition 3.9}.
If we endow \(E\) with a connection \(K_E\), the \emph{horizontal bundle} \(H_K E \to E\) is the subbundle of \(TE \to E\) defined by~\cite{gk}*{definition 3.4}
\[
  H_K E = \{ \nu \in TE \colon K_E(\nu) = 0\} = \operatorname{ker}\, K_E;
\]
then the direct decomposition \(TE = H_K E \oplus \, VE\) holds~\cite{gk}*{proposition 3.5}. 
Equivalently, there exists a map called the \emph{horizontal lift} \(\operatorname{Hor}_K \colon TM \to H_KE \subset TE\) such that for each \(e \in TM\), \((T\pi_M \circ \operatorname{Hor}_K)(e) = e\). 
Consequently, if \(\mathrm{id}_{TE} \colon TE \to TE\) is the identity map, then
\begin{equation}\label{IdonTE}
\mathrm{id}_{TE} = \operatorname{Vert} \circ \, K_E + \operatorname{Hor}_K \circ \, T\pi_M. 
\end{equation}

If \((M,g_M)\) is a Riemannian manifold and if we endow \(E\) with a metric connection \(K_E\), then the \emph{Sasaki metric} \(G^S_E\)~\cite{sasaki} (see also~\cite{docarmo}*{chapter 3, exercise 2}) is defined for every \(\nu \in TE\) by 
\[
G^S_E(\nu) = g_M (T\pi_M(\nu)) + g_E(K_E(\nu)). 
\]

If \(K_{TM}\) and \(K_{TN}\) are connections on \(TM\) and on \(TN\) respectively, a connection \(K_{T^*M\otimes TN} \colon T(T^*M\otimes TN) \to T^*M \otimes TN\) can be defined for every \(v \in T(T^*M)\) and \(w \in T(TN)\) by 
\[
K_{T^*M \otimes TN}(v \otimes w) = K_{T^*M}(v) \otimes K_{TN}(w),
\]
where \(K_{T^*M} \colon T(T^*M) \to T^*M\) is a connection on \(T^*M\) induced by \(K_{TM}\). 
If \(u\colon M \to N\) is a smooth map, the \emph{second-covariant derivative} \(D_K^2 u \colon M \to T^*M \otimes T^*M \otimes TN\) with respect to the connection \(K_{T^*M\otimes TN}\) is defined by
\[
D^2_K u = K_{T^*M\otimes TN} \circ \, T^2 u. 
\]
More specifically, for every \(x\in M\), the map \(D^2_K u(x) \colon T_x M \times T_x M \to T_{u(x)} N\) is bilinear~\cite{wendl}*{equation (3.10)} and for every \(e\in T_x M\), by equation~\eqref{IdonTE}, we have the decomposition
\[
T^2u(x)[e] = \operatorname{Vert} \circ \, D^2_K u(x)[e] + \operatorname{Hor}_K \circ \, Tu(x)[e], 
\]
where \(D^2_K u(x)[e] = K_{T^*M \otimes TN} (T^2u(x)[e]) \in \mathcal{L}(T_x M, T_{u(x)} N)\) and with the usual identification of \(M \times TN\) with \(M \times \{0\} \times TN \subseteq TM \times TN\). 

%**************************************
\subsection{Definition and properties of colocal weak covariant derivatives} 
In this part, we assume that \(M\) and \(N\) are affine manifolds. 

We first define the notion of colocal weak covariant derivatives. 
\begin{de}\label{weakSecondCovariant}
Let \(\upsilon \colon M \to T^*M \otimes TN\) be a colocally weakly differentiable map that covers a map \(u \colon M \to N\). 
A map \(D_K \upsilon \colon M \to T^*M \otimes T^*M \otimes TN\) is a \emph{colocal weak covariant derivative} of \(\upsilon\) whenever \(D_K \upsilon\) is measurable, for almost every \(x\in M\), the map \(D_K \upsilon(x) \colon T_x M \times T_x M \to T_{u(x)} N\) is bilinear and for every \(f\in C^1_c(T^*M\otimes TN,\R)\), 
\[
(Tf)(\upsilon)[\operatorname{Vert} \circ\,D_K\upsilon] 
= T(f \circ \upsilon) - (Tf)(\upsilon) [\operatorname{Hor}_K \circ\, \upsilon]
\]
almost everywhere in \(M\). 
\end{de}

Since the vertical lifting \(\operatorname{Vert}\) is injective and since \(Tf\) can be taken to be injective at a given set of points, a map \(\upsilon\) has at most one  colocal weak covariant derivative. 

The previous definition~\ref{weakSecondCovariant} requires the colocal weak differentiability of the map \(\upsilon\). 
The main result of the current section is that the colocal weak derivative and the colocal weak covariant derivative are equivalent objects.

\begin{propo}\label{covariantTTU}
Let \(u \colon M \to N\) be a twice colocally weakly differentiable map. 
Then the map \(Tu \colon M \to T^*M\otimes TN\) has a colocal weak covariant derivative \(D_K (Tu) \colon M \to T^*M \otimes T^*M\otimes TN\) 
if and only if 
the map \(Tu \colon M \to T^*M\otimes TN\) has a colocal weak derivative \(T^2u \colon TM \to T(T^*M\otimes TN)\). 
Furthermore, they are unique and almost everywhere in \(M\) 
\begin{align*}
T^2u &= \operatorname{Vert} \circ\,D_K(Tu) + \operatorname{Hor}_K \circ\,Tu, &
&\text{ and }&
D_K(Tu) &= K_{T^*M\otimes TN} \circ T^2u. 
\end{align*}
\end{propo}

Thus 
the colocal weak covariant derivatives for two connections \(K_1\) and \(K_2\) can be related to each others via the identity 
\[
  \operatorname{Vert} \circ \, D_{K_1} (Tu) + \operatorname{Hor}_{K_1} \circ \,Tu
  = \operatorname{Vert} \circ \, D_{K_2} (Tu) + \operatorname{Hor}_{K_2} \circ \,Tu.
\]

\begin{proof}[Proof of proposition~\ref{covariantTTU}]
On the one hand, we assume that the colocal weak covariant derivative \(D_K (Tu)\) exists. 
Then for every \(f\in C^1_c(T^*M\otimes TN,\R)\),
\[
T(f\circ\,Tu) = (Tf)(Tu)[\operatorname{Vert} \circ\, D_K (Tu)  + \operatorname{Hor}_K \circ\, Tu]
\]
and we can take
\[
T^2 u = \operatorname{Vert} \circ\,D_K(Tu) + \operatorname{Hor}_K \circ\, Tu
\]
for the second-order derivative of \(u\).

Conversely, we assume that there is a colocal weak derivative \(T^2u \colon TM \to T(T^*M\otimes TN)\) such that for every \(f\in C^1_c(T^*M \otimes TN, \R)\), \(T(f\circ\,Tu) = Tf \circ T^2u\) almost everywhere in \(M\). 
In view of the identity~\eqref{IdonTE}, almost everywhere in \(M\)
\[
 T (f \circ\,Tu) =  (Tf)(Tu)[\operatorname{Vert} \circ\, K_{T^*M\otimes TN} \circ T^2u  + \operatorname{Hor}_K \circ\, Tu]
\]
and we can thus take \(D_K(Tu) = K_{T^*M\otimes TN} \circ T^2 u\). 
\end{proof}

We assume that \((M,g_M)\) and \((N,g_N)\) are Riemannian manifolds with the respective Levi--Civita connection maps on \(TM\) and on \(TN\). 
Our concept of covariant derivative is, under some technical assumptions, equivalent to the notion of covariant derivative of P. Hornung and R. Moser~\cite{hm}*{definition 2.5}. 

Indeed, the metrics on vectors of \(TM\) and \(TN\) induce a metric \(g^*_M\otimes g^*_M \otimes g_N\) on \(T^*M \otimes T^*M \otimes TN\). 
This metric can be computed for every bilinear map \(\xi \colon T_x M \times T_x M \to T_y N\) by
\[ 
(g^*_M\otimes g^*_M \otimes g_N)(\xi,\xi) = \sum_{1 \le i,j \le m} g_N \bigl(\xi(e_i,e_j), \xi(e_i,e_j)\bigr),
\]
where \((e_i)_{1\le i \le m}\) is any orthonormal basis in \(\pi^{-1}_M(\{x\})\) with respect to the Riemannian metric \(g_M\). 

\begin{propo}\label{HornungMoserLink}
Let \(u \in \dot{W}^{1,1}_\mathrm{loc}(M,N)\) and let \(f \colon T^*M\otimes TN \to T^*M\otimes TN\) be defined for every \(\xi \in T^*M\otimes TN\) by 
\[
f(\xi) = \frac{\xi}{\sqrt{1+\abs{\xi}^2_{g^*_M\otimes g_N}}}.
\]
If \(f \circ\,Tu\) is colocally weakly differentiable, then \(Tu\) is colocally weakly differentiable and almost everywhere in \(M\)
\begin{align*}
D_K(Tu)
= \Bigl(\id - \frac{Tu \otimes Tu}{1+\abs{Tu}_{g^*_M\otimes g_N}^2} \Bigr)^{-1} \sqrt{1+ \abs{Tu}^2_{g^*_M \otimes g_N}} \,
D_K (f \circ\,Tu) .
\end{align*}
Conversely if \(Tu\) is colocally weakly differentiable and if \(\abs{D_K(Tu)}_{g^*_M\otimes g^*_M \otimes g_N} \in L^1_\mathrm{loc}(M)\), then \(f \circ\,Tu\) is colocally weakly differentiable.
\end{propo}

The advantage of this formulation is that \(f \circ\,Tu\) is a bounded measurable bundle morphism, and thus if \(N\) is compact, its colocal weak differentiability is equivalent to its weak differentiability in local charts or in an isometric embedding. In Sobolev spaces, the additional integrability assumption for the converse implication is automatically satisfied.

\begin{proof}[Proof of proposition~\ref{HornungMoserLink}]
Let us first assume  that the composite map \(f \circ\,Tu\) is colocally weakly differentiable. 
Since the map \(f\) is invertible on its image, for every \(h \in C^1_c(T^*M\otimes TN,\R)\), 
\[
h \circ\,Tu = (h\circ f^{-1}) \circ (f\circ\,Tu),
\]
where the map \(h \circ f^{-1}\) can be extended to a compactly supported map. 
By the definition of colocal weak differentiability of \(f \circ\,Tu\), the right-hand side is weakly differentiable and thus the left-side is also weakly differentiable and the map \(Tu\) is colocally weakly differentiable by definition. In particular we obtain that 
\[
  T f \circ T^2 u = T (f \circ\,Tu),
\]
and the identity follows. Since the connection \(K_{T^*M\otimes TN}\) is metric, by geometric properties of \(f\) and uniqueness of colocal weak covariant derivatives, we can deduce the desired formula.

Conversely, almost everywhere in \(M\)
\begin{equation*}
\abs{T^2u}^2_{g^*_M \otimes G^S_{T^*M\otimes TN}} = \abs{Tu}_{g^*_M\otimes g_N}^2 + \abs{D_K (Tu)}_{g^*_M\otimes g^*_M \otimes g_N}^2. 
\end{equation*}
Hence, \(Tu \in \dot{W}^{1,1}_\mathrm{loc}(M,T^*M\otimes TN)\) and for every \(h\in C^1_c(T^*M \otimes TN, \R)\), since the map \( h \circ f \in C^1(T^*M\otimes TN, \R)\) is Lipschitz-continuous, by the chain rule in Sobolev spaces for maps between manifolds~\cite{cvs}*{proposition 2.6}, the map \((h \circ f) \circ\,Tu\) is weakly differentiable, and so \(f \circ\,Tu\) is colocally weakly differentiable. 
\end{proof}

%*********************************************
\subsection{Sequences of second order Sobolev maps} 
We now define intrinsic second-order Sobolev spaces. 

\begin{de}\label{defW2p}
Let \(p \in [1,+\infty)\). 
A map \(u \colon M \to N\) belongs to the \emph{second-order Sobolev space} \(\dot{W}^{2,p}(M,N)\) whenever the map \(u\) is twice colocally weakly differentiable and \\ \(\abs{D_K (Tu)}_{g^*_M \otimes g^*_M \otimes g_N} \in L^p(M)\). 
\end{de}

A classical technique in the calculus of variations is to extract from a minimizing sequence a subsequence that converges almost everywhere. 
For that, we have a Rellich--Kondrashov type compactness theorem as follows. 
\begin{propo}[Rellich--Kondrashov for second-order Sobolev maps]\label{RellichSobolev}
Let \((u_\ell)_{\ell \in \N}\) be a sequence of twice colocally weakly differentiable maps from \(M\) to \(N\) and let \(v \in \dot{W}^{1,p}(M,N)\).
If \((N, d)\) is complete, if there exists \(p \in [1, +\infty)\) such that 
\[
  \sup_{\ell \in \N} \int_{M} d (u_\ell, v)^p + \abs{T u_\ell}_{g^*_M\otimes g_N}^p + \abs{D_K(Tu_\ell)}_{g^*_M\otimes g^*_M \otimes g_N}^p < +\infty,
\]
then there is a subsequence \((Tu_{\ell_k})_{k \in \N}\) that converges to a measurable map \(\upsilon \colon M \to T^*M \otimes TN\) almost everywhere in \(M\).
\end{propo}

In particular, if the map \(\upsilon\) covers the map \(u\colon M \to N\), the subsequence \((u_{\ell_k})_{k\in \N}\) converges to \(u\) almost everywhere in \(M\). 

For \(p \in (1,2)\), P. Hornung and R. Moser have a similar result~\cite{hm}*{lemma 3.2} for their notion of second-order Sobolev spaces~\cite{hm}*{definition 3.1}.

\begin{proof}[Proof of proposition~\ref{RellichSobolev}]
Since \((N,d)\) is complete, the space \((T^*M\otimes TN, d^S)\), where \(d^S\) is the distance induced by the Sasaki metric \(G^S_{T^*M\otimes TN}\), is complete.
Next, there exists \(C >0\) such that 
\[
\sup_{\ell \in \N} \int_M d^S(Tu,Tv)^p \le C \sup_{\ell \in \N}\int_M d(u_\ell,v)^p + \abs{Tu_\ell}_{g^*_M\otimes g_N}^p + \abs{Tv}_{g^*_M\otimes g_N}^p < +\infty.
\]
Since for every \(\ell \in \N\), 
\[
\abs{T^2u_\ell}^2_{g^*_M \otimes G^S_{T^*M\otimes TN}} 
= \abs{Tu_\ell}_{g^*_M\otimes g_N}^2 + \abs{D_K (Tu_\ell)}_{g^*_M\otimes g^*_M \otimes g_N}^2 
\]
almost everywhere in \(M\), 
we have 
\[
\sup_{\ell\in \N} \int_M \abs{T^2u_\ell}^p < +\infty.  
\]
Hence, the sequence \((Tu_{\ell})_{\ell\in \N}\) satisfies all the assumptions of the Rellich--Kondrashov compactness property for Sobolev maps~\cite{cvs}*{proposition 3.4} and so there exists a subsequence \((Tu_{\ell_k})_{k \in \N}\) that converges to a measurable map \(\upsilon \colon M \to T^*M\otimes TN\) almost everywhere in \(M\).
\end{proof}

Before considering sequences of Sobolev maps, we rephrased the closure property (proposition~\ref{closurePropertyTwice}) with the notion of colocal weak covariant derivative. 

\begin{propo}[Closure property]\label{closureProperty}
Let \((u_\ell)_{\ell \in \N}\) be a sequence of twice colocally weakly differentiable maps from \(M\) to \(N\). 
If the sequence \((Tu_\ell)_{\ell \in \N}\) converges locally in measure to a measurable map \(\upsilon \colon M \to T^*M \otimes TN\) that covers a map \(u\colon M \to N\), if the sequences \((Tu_\ell)_{\ell \in \N} \) and \((D_K (Tu_\ell))_{\ell \in \N}\) are bilocally uniformly integrable, then the map \(u\) is twice colocally weakly differentiable. 
\end{propo}

It is a direct consequence of the following lemma and proposition~\ref{closurePropertyTwice}.

\begin{lemme}
Let \((u_\ell)_{\ell \in \N}\) be a sequence of twice colocally weakly differentiable maps from \(M\) to \(N\). 
Then the sequence \((T^2 u_\ell)_{\ell \in \N}\) is bilocally uniformly integrable if and only if 
the sequence \((D_K (Tu_\ell))_{\ell \in \N}\) is bilocally uniformly integrable. 
\end{lemme}

\begin{proof}
For every \(y\in T^*M \otimes TN\), there exist an open set \(U \subseteq T^*M \otimes TN\) such that \(y\in U\) and a map \(\varphi \in C^1_c(T^*M\otimes TN,\R^q)\) such that \(\varphi_{\arrowvert U} \colon U \subseteq T^*M\otimes TN \to \R^q\) is a local chart~\cite{cvs}*{lemma 1.6} (see also lemma~\ref{lemmeExtended} above). 
Hence, for every local chart \(\psi \colon V \subseteq M \to \R^m\), every \(\ell \in \N\) and almost everywhere on \(\psi(V \cap (T u_\ell)^{-1}(U))\), by the relation between \(T^2u_\ell\) and \(D_K(T u_\ell)\) (proposition~\ref{covariantTTU}),
\[
T\varphi \circ T^2 u_\ell \circ T(\psi^{-1}) 
= T\varphi \circ \operatorname{Vert} \circ\,D_K (T u_\ell) \circ T(\psi^{-1}) 
+ T\varphi \circ \operatorname{Hor}_K \circ \,T u_\ell \circ T(\psi^{-1}). 
\]
Since \(\supp(\varphi)\) is compact, the second term is uniformly bounded for \(\ell\in \N\), and so the conclusion follows using this equality and one or an other assumption. 
\end{proof}

Assuming that there exists a subsequence that converges almost everywhere, it is important to have some closure property in the particular case of bounded sequences in Sobolev spaces. 

\begin{propo}[Weak closure property for second-order Sobolev spaces]\label{closurePropertySobolev}
Let \(p \in [1,+\infty)\). 
Let \((u_\ell)_{\ell \in \N}\) be a sequence of twice colocally weakly differentiable maps from \(M\) to \(N\). 
Assume that the sequence \((Tu_\ell)_{\ell\in \N}\) converges locally in measure to a measurable map \(\upsilon \colon M \to T^*M\otimes TN\) that covers a map \(u \colon M \to N\), that
\[
\limsup_{\ell \to \infty} \int_M \left(\abs{Tu_\ell}_{g^*_M\otimes g_N}^p + \abs{D_K (Tu_\ell)}_{g^*_M \otimes g^*_M\otimes g_N}^p \right) < +\infty,
\]
and, if \(p=1\), that the sequences \((Tu_\ell)_{\ell \in \N}\) and \((D_K (Tu_\ell))_{\ell \in \N}\) are bilocally uniformly integrable. 
Then \(u\in \dot{W}^{2,p}(M,N)\) and
\[
\int_M \abs{D_K (Tu)}_{g^*_M\otimes g^*_M \otimes g_N}^p \le \liminf_{\ell \to \infty} \int_M \abs{D_K (Tu_{\ell})}_{g^*_M\otimes g^*_M \otimes g_N}^p. 
\]
\end{propo}

For \(p \in (1,2)\), P. Hornung and R. Moser have a similar result~\cite{hm}*{lemma 3.2} for their notion of second-order Sobolev spaces~\cite{hm}*{definition 3.1}.

We first prove the following lemma.
\begin{lemme}\label{absTu}
Let \(p \in [1,+\infty)\). 
If \(u \in \dot{W}^{1,p}(M,N) \cap \dot{W}^{2,p}(M,N)\), then the map \(\abs{Tu}_{g^*_M \otimes g_N}\) belongs to \(W^{1,p}(M)\) and almost everywhere in \(M\)
\[
\bigabs{\,T\abs{Tu}_{g^*_M\otimes g_N}}_{g^*_M \otimes g_1} \le \abs{D_K(Tu)}_{g^*_M \otimes g^*_M \otimes g_N},
\]
where \(g_1\) is the Euclidean metric on \(\R\). 
\end{lemme}

\begin{proof}
Since \(Tu \in \dot{W}^{1,p}(M,T^*M\otimes TN)\) (see proof of proposition~\ref{HornungMoserLink} above) and since the map \(\abs{\cdot}_{g^*_M \otimes g_N}\) is Lipschitz-continuous with respect to the distance induced by the Sasaki metric \(G^S_{T^*M\otimes TN}\), by the chain rule in Sobolev spaces between manifolds~\cite{cvs}*{proposition 2.6}, \(\abs{Tu}_{g^*_M\otimes g_N} \in W^{1,p}(M)\). 
Since the connection \(K_{T^*M\otimes TN}\) is metric, by the chain rule formula in Sobolev spaces~\cite{cvs2}*{proposition 5.2} and by the relation between \(T^2 u\) and \(D_K(Tu)\) (proposition~\ref{covariantTTU}), we have the desired inequality. 
\end{proof}

\begin{proof}[Proof of proposition~\ref{closurePropertySobolev}]
By weak closure property for Sobolev maps~\cite{cvs}*{proposition 3.8}, \(u \in \dot{W}^{1,p}(M,N)\) and \(Tu = \upsilon\). 
Moreover, by closure property in terms of the colocal weak covariant derivative (proposition~\ref{closureProperty}), the map \(u\) is twice colocally weakly differentiable. 
We now need to prove that \(\abs{D_K (Tu)}_{g^*_M\otimes g^*_M \otimes g_N} \in L^p(M)\). 

By previous lemma~\ref{absTu}, the sequence \((\abs{Tu_\ell}_{g^*_M\otimes g_N})_{\ell \in \N}\) is bounded in \(W^{1,p}(M)\). 
By the classical Rellich--Kondrashov compactness theorem~\citelist{\cite{aubin}*{theorem 2.34 (a)}\cite{brezis}*{theorem 9.16}}, and since \((\abs{Tu_{\ell}}_{g^*_M \otimes g_N})_{\ell \in \N}\) converges to \(\abs{Tu}_{g^*_M \otimes g_N}\) in measure, the sequence \((\abs{Tu_{\ell}}_{g^*_M \otimes g_N})_{\ell \in \N}\) converges to \(\abs{Tu}_{g^*_M \otimes g_N}\) in \(L^p_\mathrm{loc}(M)\). 
Up to a subsequence, the sequence \((\abs{D_K (Tu_{\ell})}_{g^*_M \otimes g^*_M\otimes g_N})_{\ell\in \N}\) converges weakly to some \(w\) in \(L^p_\mathrm{loc}(M)\). 
Since the sequence \((Tu_{\ell})_{\ell\in \N}\) converges to \(Tu\) locally in measure, for every \(f\in C^1_c(T^*M\otimes TN,\R^q)\), 
with \(q \ge \min(m, \mathrm{dim}(T^*M\otimes TN))\), 
\[
\abs{T(f\circ\,Tu)}_{g^*_M\otimes g_q}^2 \le \abs{f}_{\mathrm{Lip}}^2 (w^2 + \abs{Tu}^2_{g^*_M\otimes g_N}),
\]
where \(g_q\) is the Euclidean metric on \(\R^q\), almost everywhere in \(M\). 
By the characterization of the norm of the derivative~\cite{cvs}*{proposition 2.2}, 
\[
\abs{Tu}_{g^*_M\otimes g_N}^2 + \abs{D_K (Tu)}_{g^*_M\otimes g^*_M \otimes g_N}^2 \le w^2 + \abs{Tu}^2_{g^*_M\otimes g_N}
\]
almost everywhere in \(M\) and so \(\abs{D_K (Tu)}_{g^*_M\otimes g^*_M \otimes g_N}\)\(\le w\).
Consequently, by lower semi-continuity of the norm under weak convergence, for each compact set \(Q \subseteq M\), 
\[
\int_Q \abs{D_K (Tu)}_{g^*_M\otimes g^*_M \otimes g_N}^p \le \int_Q w^p \le \liminf_{\ell\to \infty} \int_Q \abs{D_K (Tu_{\ell})}_{g^*_M\otimes g^*_M \otimes g_N}^p. 
\]
Finally, since \(M\) has a countable basis, there exists a set \(\{Q_i \colon i \in \N\}\) of compact sets such that for every \(i\in \N\), \(Q_i \subseteq Q_{i+1}\) and \(M = \bigcup_{i\in \N} Q_i\)~\cite{lee}*{proposition 4.76} 
and so by the monotone convergence theorem~\cite{bogachev}*{theorem 2.8.2}, we have the desired inequality. 
\end{proof}

%********************************************
\subsection{Higher order weak covariant derivatives and Sobolev spaces}
First, we assume that \(M\) and \(N\) are affine manifolds. 
In order to simplify the notations, we denote indifferently the connections \(K\), the maps \(\operatorname{Vert}\) and \(\operatorname{Hor}_K\) related to the different bundles (section~\ref{sectionPreliminaries}). 

If \(k=2\) and \(u\) is twice colocally weakly differentiable, then \(D^1_K u = Tu\) and \(D^2_K u = D_K (Tu)\). As for the notion of twice colocally weakly differentiable maps, there is a relation between weak higher-order covariant derivative and higher-vector bundle morphism. 

\begin{propo}\label{covariantTkU}
Let \(u \colon M \to N\) be an \(k\) times colocally weakly differentiable map. 
Then the colocally weakly differentiable map \(D_K^{k - 1} u \colon M \to (\otimes^{k-1} T^*M) \otimes TN\) has a colocal weak covariant derivative \(D_K^k u \colon M \to T^*M \otimes (\otimes^{k-1} T^*M) \otimes TN\) 
if and only if the map \(D_K^{k - 1} u \colon M \to (\otimes^{k-1} T^*M) \otimes TN\) has a colocal weak derivative \(T D_K^{k - 1} u \colon TM \to T((\otimes^{k-1} T^*M) \otimes TN)\), and almost everywhere in \(M\)
\[
 T D_K^{k - 1} u = \operatorname{Vert} \circ\,D_K^{k} u + \operatorname{Hor}_K \circ\,D_K^{k-1} u,
\]
and 
\[
D_K^k u = K \circ T D_K^{k-1} u. 
\]
Moreover, for every \(j \in \N_*\),  the map \(D_K^{k} u\) is \(j\) times colocally weakly differentiable if and only if 
the map \(D_K^{k-1} u\) is \((j + 1)\) times colocally weakly differentiable, and almost everywhere in \(M\) 
\begin{equation}
\label{eqMasterHigherOrder}
T^{j + 1} D_K^{k - 1} u 
= (T^j\operatorname{Vert}) \circ (T^j D_K^{k} u) + (T^j \operatorname{Hor}_K) \circ (T^j D_K^{k-1} u), 
\end{equation}
and 
\begin{equation}\label{eqMasterHigherOrderBis}
T^j D_K^k u = (T^j K) \circ (T^{j+1} D_K^{k-1} u). 
\end{equation}
\end{propo}

In a second step, we assume that \((M,g_M)\) and \((N,g_N)\) are Riemannian manifolds with Levi--Civita connection maps respectively. The metrics on vectors of \(TM\) and \(TN\) induce a metric \((\otimes^k g^*_M) \otimes g_N\) on \((\otimes^k T^*M) \otimes TN\). This metric can be computed for every \(k\)-linear map \(\xi \colon \times^k T_x M \to T_y N\) by 
\[
\bigl((\otimes^k g^*_M) \otimes g_N\bigr)(\xi,\xi) = \sum_{1 \le i_1,\dots, i_k \le m} g_N \bigl(\xi(e_{i_1},\dots, e_{i_k}), \xi(e_{i_1},\dots, e_{i_k})\bigr),
\]
where \((e_i)_{1\le i \le m}\) is an orthonormal basis in \(\pi^{-1}_M(\{x\})\) with respect to the Riemannian metric \(g_M\). 

We are now able to define higher-order Sobolev spaces. 
\begin{de}\label{defWkp}
Let \(p \in [1,+\infty)\). 
A map \(u \colon M \to N\) belongs to the \emph{\(k^{\text{th}}\) order Sobolev space} \(\dot{W}^{k,p}(M,N)\) whenever \(u\) is \(k\) times colocally weakly differentiable and \(\abs{D^k_K u}_{(\otimes^k g^*_M) \otimes g_N} \in L^p(M)\). 
\end{de}

First, we have a Rellich--Kondrashov type compactness theorem. 
\begin{propo}
Let \((u_\ell)_{\ell \in \N}\) be a sequence of \(k\) times colocally weakly differentiable maps from \(M\) to \(N\) and let \(v \in \bigcap_{j=1}^{k-1} \dot{W}^{j,p}(M,N)\). 
If \((N,d)\) is complete, if there exists \(p \in [1,+\infty)\) such that 
\[
  \sup_{\ell \in \N} \int_{M} d (u_\ell, v)^p + \sum_{j=1}^k \abs{D^j_K u_\ell}_{(\otimes^j g^*_M) \otimes g_N}^p < +\infty,
\]
then there is a subsequence \((T^{k-1} u_{\ell_i})_{i\in \N}\) that converges to a measurable map \(\upsilon \colon M \to \Mor(T^{k-1} M,T^{k-1} N)\) almost everywhere in \(M\). 
\end{propo}

By relying recursively on the formulas~\eqref{eqMasterHigherOrder} and~\eqref{eqMasterHigherOrderBis}, the boundedness of \((D_{K}^k u_\ell)_{\ell \in \N}\) is related to the boundedness of the sequence \((T^k u_\ell)_{\ell \in \N}\) and 
thus, the sequence \((T^{k-1} u_\ell)_{\ell \in \N}\) satisfies all the assumptions of the Rellich--Kondrashov compactness property for Sobolev maps~\cite{cvs}*{proposition 3.4} and so the conclusion follows directly. 

We also have a closure property in the particular case of bounded sequences in Sobolev spaces. 
\begin{propo}[Weak closure property for higher-order Sobolev spaces]
Let \(p \in [1,+\infty)\). 
Let \((u_\ell)_{\ell \in \N}\) be a sequence of \(k\) times colocally weakly differentiable maps from \(M\) to \(N\). 
Assume that the sequence \((T^{k-1} u_\ell)_{\ell\in \N}\) converges locally in measure to a measurable map \(\upsilon \colon M \to \Mor(T^{k-1} M,T^{k-1}N)\) that covers a map \(u \colon M \to N\), that
\[
\limsup_{\ell \to \infty} \int_M \sum_{j=1}^k \abs{D^j_K u_\ell}_{(\otimes^j g^*_M) \otimes g_N}^p < +\infty,
\]
and, if \(p=1\), that for every \(j\in \{1,\dots,k\}\), the sequence \((D^j_K u_\ell)_{\ell \in \N}\) is bilocally uniformly integrable. 
Then \(u\in \dot{W}^{k,p}(M,N)\) and
\[
\int_M \abs{D^k_K u}_{(\otimes^k g^*_M) \otimes g_N}^p \le \liminf_{\ell \to \infty} \int_M \abs{D^k_K u_{\ell}}_{(\otimes^k g^*_M) \otimes g_N}^p. 
\]
\end{propo}

By relying recursively on the formulas~\eqref{eqMasterHigherOrder} and~\eqref{eqMasterHigherOrderBis} and applying the weak closure property for Sobolev maps~\cite{cvs}*{proposition 3.8} to the sequence \((T^{k-1} u_\ell)_{\ell \in N}\), we have \(u \in \dot{W}^{k-1,p}(M,N)\) and \(T^{k-1} u = \upsilon\). For the last inequality, we can proceed as in the proof of proposition~\ref{closurePropertySobolev}, and so by first proving a lemma similar to lemma~\ref{absTu}. 

\begin{lemme}\label{absDku}
Let \(p\in [1,+\infty)\). If \(u \in \dot{W}^{k-1,p}(M,N) \cap \dot{W}^{k,p}(M,N)\), then the map \(\abs{D^{k-1}_K u}_{(\otimes^{k-1} g^*_M) \otimes g_N}\) belongs to \(W^{1,p}(M)\) and almost everywhere in \(M\) 
\[
\bigabs{\,T \abs{D^{k-1}_K u}_{(\otimes^{k-1} g^*_M) \otimes g_N}}_{g^*_M \otimes g_1} \le \abs{D^k_K u}_{(\otimes^k g^*_M) \otimes g_N},
\]
where \(g_1\) is the Euclidean metric on \(\R\). 
\end{lemme}

%******************************************
\section{Chain rule for higher order colocally weakly differentiable maps}\label{sectionChainRule} 

Since the chain rule is central in the definition of first-order colocal weak differentiability and since composition is a crucial tool in the theory of Sobolev maps, in particular, in the definition by embedding~\eqref{Wkp}, we investigate under which condition the chain rule holds for higher-order colocally weakly differentiable maps: whether
for a \(k\) times colocally weakly differentiable map \(u \colon M \to N\) and for a function \(f\in C^k_c(N,\R)\), the composite function \(f\circ u\) is \(k\) times weakly differentiable and whether we have \(T^k(f\circ u) = T^k f \circ T^k u\) almost everywhere on the manifold \(M\).

%*****************************************
\subsection{Failure of the chain rule for twice colocally weakly differentiable maps}

The starting point of the analysis is that that the higher-order chain rule \emph{does not hold} in general. 
\begin{example}\label{exampleNoChainRule}
Let \(m \in \N\) with \(m \ge 2\) and let \(\alpha > 0\). We define the function \(u \colon \R^m\setminus\{0\} \to \R^2\) for each \(x\in \R^m\setminus\{0\}\) by 
\begin{equation*}
  u(x) = \bigl(\cos(\abs{x}^{-\alpha}), \sin(\abs{x}^{-\alpha})\bigr).
\end{equation*}
For every \(x\in \R^m\setminus\{0\}\), \(\abs{Du(x)}_{g^*_m \otimes g_2} = \alpha\abs{x}^{-\alpha -1}\) and \(\abs{D^2 u(x)}_{g^*_m \otimes g^*_m \otimes g_2} \ge \alpha\abs{x}^{-2\alpha-2}\). 
If \(\alpha-1 < m < 2(\alpha +1)\), then the map \(u\) is twice colocally weakly differentiable but if we take \(f\in C^2_c(\R^2, \R^2)\) such that \(f = \mathrm{id}\) on \(B_2\), \(f\circ u = u \) on \(\R^m \setminus\{0\}\) and so \(\abs{D^2(f\circ u)}_{g^*_m \otimes g^*_m \otimes g_2} = \abs{D^2 u}_{g^*_m \otimes g^*_m \otimes g_2}\) does not belong to \(L^1_{\mathrm{loc}}(\R^m)\).
\end{example}

The crucial point in this example is that \(\abs{D^2 u}\) is integrable on bounded sets on which the derivative \(\abs{Du}\) is bounded, but is not integrable on bounded sets on which the map \(u\) itself is bounded.

%***************************************
\subsection{Double norms for double vector bundles}

In order to characterize the maps for which a chain rule holds, we introduce a notion 
of norm for colocal weak derivatives of a bundle morphism. 
Since \(\Mor(T^2M, T^2N)\) does not carry a vector bundle structure over \(M\times N\), we cannot define a notion of norm on this space with respect to each fiber over \(M\times N\). 
However, we define a notion of \emph{double norm} on this space which is compatible with the double vector bundle structure. 

To fix the idea, since \(\pi^2_M \colon T^2M \to M\) is not a vector bundle~\citelist{\cite{bertram}*{\S 9.3}\cite{dieudonne}*{\S 16.15.7}} and \(T^2M\) is a double vector bundle, we first define this notion on \(T^2M\).
\begin{de}\label{defDoubleSeminorm}
A map \(\double{\cdot} \colon T^2 M \to \R\) is a \emph{double seminorm} on \(T^2 M\) whenever 
\begin{enumerate}[(a)]
\item the map \(x\in M \mapsto \double{\cdot}_x\) is continuous, where for every \(x\in M\), \(\double{\cdot}_x\) is the restriction of \(\double{\cdot}\) to the space \(T_x^2 M\), 
\item \(\double{\cdot}\) is a seminorm on fibers over both the bundles \((T^2M, T\pi_M, TM)\) and \((T^2M, \pi_{TM},\) \(TM)\), that is, for every \(e \in TM\), the restriction of \(\double{\cdot}\) to the fibers \((T\pi_M)^{-1}(\{e\})\) or \((\pi_{TM})^{-1}(\{e\})\) is a seminorm. 
\end{enumerate}
\end{de}

Consequently, if \(\double{\cdot}\) is a double seminorm on \(T^2M\), for every \(\nu, \mu \in T^2M\) such that \(\pi_{TM}(\nu) = \pi_{TM}(\mu)\) and every \(\lambda \in \R\), 
\[
\double{\lambda \cdot_{\pi_{TM}} \nu} = \abs{\lambda} \double{\nu},
\]
where \(\cdot_{\pi_{TM}}\) is the vector bundle multiplication of \((T^2M, \pi_{TM}, TM)\)~\cite{mackenzie}*{\S 1}, and 
\[
\double{\nu +_{\pi_{TM}} \mu} \le \double{\nu} + \double{\mu},
\]
where \(+_{\pi_{TM}}\) is the vector bundle addition of \((T^2M, \pi_{TM}, TM)\)~\cite{mackenzie}*{\S 1}; and it satisfies the same properties if we consider instead the vector bundle operations \(\cdot_{T\pi_M}\) and \(+_{T\pi_M}\) of \((T^2M, T\pi_M, TM)\).

\begin{de}\label{defDoubleNorm}
A map \(\double{\cdot} \colon T^2 M \to \R\) is a \emph{double norm} on \(T^2 M\) whenever 
\begin{enumerate}[(a)]
\item \(\double{\cdot}\) is a double seminorm on \(T^2 M\),
\item \(\double{\cdot}\) is maximal: for every double seminorm \(\double{\cdot}'\) on \(T^2 M\), there exists a positive continuous function \(\beta \in C(M,\R)\) such that for every \(\nu \in T^2 M\), 
\[
 \double{\nu}' \le \beta (\pi^2_M(\nu)) \double{\nu}. 
\]
\end{enumerate}
\end{de} 

The maximality implies that all double norms are locally equivalent. 
\begin{propo}\label{existsDoubleNormT2M}
There exists a double norm on \(T^2 M\).
\end{propo}

\begin{proof}[Sketch of the proof]
If \(M = \R^m\), then \(T^2 M = \R^m \times \R^m \times \R^m\times \R^m\) and a canonical double seminorm \(\double{\cdot}_* \colon T^2M \to \R\) can be defined for every \(\nu = (x, e_1, e_2, e_{12}) \in T^2 M\) by 
\[
  \double{\nu}_* = \abs{e_1}\,\abs{e_2} + \abs{e_{12}}. 
\]
One concludes by noting that for any double seminorm \(\double{\cdot}\) on \(T^2M\), there exists \(C >0\) such that for every \(\nu\in T^2M\), \(\double{\nu} \le C \double{\nu}_*\). 

For a general manifold \(M\), since for every \(x\in M\), there exists a local trivialization \((V,\psi)\) such that \(x\in V\), the set \(V \subseteq M\) is open and 
\[
\psi \colon V \times \R^m \times \R^m \times \R^m \to (\pi^2_M)^{-1}(V) \subseteq T^2M
\]
is a diffeomorphism~\cite{dieudonne}*{\S 16.15.7}, the argument above gives in each local trivialization a canonical double seminorm. 
These local double seminorms can then be patched together by a partition of unity.
\end{proof}

\begin{rem}
The reader will observe that a double norm is a norm on fibers over non zero elements.
This implies that if \(\nu \in T^2 M\) satisfies \(\kappa_M (\nu) = \nu\) and if 
\(\double{\nu} = 0\), then \(\nu = 0_{\pi^2_M(\nu)}\), 
where \(0_{\pi^2_M(\nu)}\) is the double zero, that is, the vector such that \(0_{\pi_{TM}}(\nu)= 0_{T\pi_M}(\nu)\)~\cite{mackenzie}*{definition 1.1}.
We will not rely on these properties because they fail for \(k^\textrm{tuple}\) norms on \(k^\text{tuple}\) vector bundles, see section~\ref{higherOrderChainRule} below.
\end{rem}

Since \(\Mor(T^2M, T^2N)\) has a double vector bundle structure, the definition and the existence of double norms can be obtained in a similar way on \(\Mor(T^2M, T^2N)\).

Motivated by the chain rule, we just remark how double norms behaves under composition of double vector bundle morphisms. 
\begin{propo}\label{compoDoubleVectorMorphism}
Let \(K\) be a manifold of class \(C^2\). 
For every compact subsets \(Q_M \subseteq M\), \(Q_N \subseteq N\) and \(Q_K \subseteq K\), there exists a constant \(C > 0\) such that for every \(\upsilon\in (\pi^2_{M\times N})^{-1}(Q_M \times Q_N)\) and every \(\xi \in (\pi^2_{N\times K})^{-1}(Q_N \times Q_K)\) such that \(\pi_N \circ \pi_{M\times N} \circ \pi^1_{T^*M\otimes TN} \circ\) \(\upsilon = \pi_N \circ \pi_{N\times K} \circ \pi^1_{T^*N\otimes TK} \circ \xi\), then 
\[
\double{\xi \circ \upsilon} 
\le C \double{\xi} \double{\upsilon}.
\]
\end{propo}

\begin{proof}[Sketch of the proof]
If \(M =\R^m\) and \(N =\R^n\), a canonical double seminorm \(\double{\cdot}_*\) on \(\Mor(T^2M, T^2N)\) can be defined for every \(f\in \Mor(T^2M,T^2N)\) by
\[
\double{f}_* = \abs{f_1}\, \abs{f_2} + \abs{f_{12}} + \abs{\hat{f}_{12}}
\]
since for every \(\nu =(x,e_1,e_2,e_{12}) \in T^2M\), 
\(
f(\nu) = (y, f_1[e_1], f_2[e_2], f_{12}[e_{12}] + \hat{f}_{12}[e_1,e_2]),
\) 
where \(f_1,f_2,f_{12} \colon \R^m \to \R^n\) are linear and \(\hat{f}_{12} \colon \R^m \times \R^m \to \R^n\) is bilinear. 
Then if \(K= \R^k\), for every \(f\in \Mor(T^2M,T^2N)\) and \(h\in \Mor(T^2N, T^2K)\),  we have 
\[
\double{h \circ f}_* \le \double{h}_* \double{f}_*.
\]

For any manifolds \(M\), \(N\) and \(K\), it suffices to prove the composition property in each local trivialization by using the one of the canonical double seminorm and then the maximal property of double norms.  
\end{proof}

%***************************************
\subsection{Chain rule for second order colocally weakly differentiable maps}
In order to have a chain rule for twice colocally weakly differentiable maps, we assume an additional integrability condition on colocal weak derivatives of bundle morphisms in terms of a double norm. We also identify the class of maps to compose with the colocal weak derivative in order to have a chain rule. 
 
For every map \(h\colon T^*M\otimes TN \to\R^q \), we denote by \(\supp_N (h)\) the \emph{support of \(h\) with respect to \(N\)}, that is, 
\[
\begin{split}
\supp_N(h) 
& = \overline{\bigl\{ y \in N \colon \text{ there exists \(x\in M\) such that } h_{\arrowvert_{\left(\pi_{M\times N}\right)^{-1}\left((x,y)\right)}} \neq 0\bigr\}}\\
& =\pi_N\bigl( \pi_{M \times N} (\supp(h))\bigr).
\end{split}
\]

\begin{propo}\label{equivalentChainRule}
Let \(u \colon M \to N\) be a colocally weakly differentiable map. 
The following statements are equivalent. 
\begin{enumerate}[(i)]
\item\label{chainRule} 
For every \(f\in C^2_c(N,\R)\), the map \(f\circ u\) is twice weakly differentiable,
\item \label{compactMorphism} 
for every bundle morphism \(h \in C^1(T^*M\otimes TN, \R)\) such that \(\supp_N (h)\) is compact, the map \(h \circ\,Tu\) is weakly differentiable,
\item\label{twice} 
the map \(u\) is twice colocally weakly differentiable and for all compact subsets \(K\subseteq M\) and \(L\subseteq N\) and every double norm \(\double{\cdot}\) on \(\Mor(T^2M, T^2N)\),
\begin{equation*}
\int_{u^{-1}(L) \cap K} \double{T^2 u} < + \infty,
\end{equation*}
\end{enumerate}
\end{propo}

We prove an intermediate result in order to prove this proposition. 
\begin{lemme}\label{likeGN}
If \(v \in W^{2, 1}_{\mathrm{loc}} (M, \R^q) \cap L^\infty(M,\R^q)\), if the bundle morphism \(w \in C^1(T^*M \otimes T\R^q,\) \(\R)\) is such that \(\supp_{\R^q} (w)\) is compact, then \(w \circ\, Tv\) is weakly differentiable. 
\end{lemme}

We recall that the \emph{operator norm} is defined for every \(\xi \in \mathcal{L}(T_x M, T_y N)\) by 
\[
\abs{\xi}_\mathcal{L} = \sup \{ \abs{\xi(e)}_{g_N} \colon e \in T_x M, \, \abs{e}_{g_M} \le 1\},
\]
where \(g_M\) and \(g_N\) are some Riemannian metrics on \(TM\) and \(TN\) respectively~\cite{docarmo}*{\S 1.2 proposition 2.10}.

\begin{proof}[Proof of lemma~\ref{likeGN}]
Without loss of generality, we can assume that \(q=1\). 

First, we assume that \(M = \Omega\) is an open set of \(\R^m\). 
By regularization theorem~\citelist{\cite{brezis}*{theorem 4.22}\cite{bogachev}*{theorem 4.2.4}\cite{willem}*{theorem 4.3.9}}, there exists a sequence \((v_\ell)_{\ell\in \N}\) in \(C^\infty_c(\Omega,\R)\) that converges to \(v\) in \(W^{2,1}_\mathrm{loc}(\Omega)\). 
Moreover, there exists \(C >0\) such that for every \(\ell \in \N\), almost everywhere in \(\Omega\)
\[
\abs{T(w\circ T v_\ell)}_{\mathcal{L}} \le C \bigl(\abs{D v_\ell}_{g^*_m \otimes g_q}^2 + \abs{D^2 v_\ell}_{g^*_m \otimes g^*_m \otimes g_q} + \abs{Dv_\ell}_{g^*_m \otimes g_q}\bigr),
\]
where \(g_m\) and \(g_q\) are the Euclidean metrics on \(\R^m\) and \(\R^q\) respectively. 
By the classical Gagliardo--Nirenberg  inequality~\citelist{\cite{gagliardo}\cite{nirenberg}}, we have \(\abs{Dv}^2_{g^*_m \otimes g_q} \in L^1_\mathrm{loc}(\Omega)\) and so the sequence \((\abs{D v_\ell}^2_{g^*_m \otimes g_q})_{\ell\in \N}\) converges to \(\abs{Dv}^2_{g^*_m \otimes g_q} \) in \(L^1_\mathrm{loc}(\Omega)\). 
Hence, by Lebesgue's dominated convergence theorem, the sequence \( (\abs{T(w\circ T v_\ell)}_\mathcal{L})_{\ell\in \N}\) converges in \(L^1_\mathrm{loc}(\Omega)\) and then by closing lemma~\cite{willem}*{lemma 6.1.5}, the map \(w \circ\, Tv\) is weakly differentiable.

For a general manifold \(M\), we apply the previous argument to every \(v\circ \psi^{-1} \colon \psi(V) \subseteq \R^m \to \R\) with any local chart \(\psi \colon V \subseteq M \to \R^m\). By a direct covering argument, we have thus that \(w \circ\, Tv\) is weakly differentiable. 
\end{proof}

\begin{proof}%
[Proof of proposition~\ref{equivalentChainRule}]
%-----
We first prove that~\eqref{chainRule} implies~\eqref{compactMorphism}.
Let \((U_i)_{i\in I}\) be an open cover of the target manifold \(N\) by sets given by lemma~\ref{lemmeExtended}. 
Let \(h \in C^1(T^*M\otimes TN, \R)\) be a bundle morphism such that \(\supp_N (h)\) is compact. 
Then there exist \(\ell \in \N_*\) such that \(\supp_N (h) \subseteq \bigcup_{i=1}^\ell U_i\) and a partition of unity \((\eta_i)_{1\le i \le \ell}\)\ subordinate to the family \((U_i)_{1\le i \le \ell}\) such that
\[
h\circ T u = \sum_{i=1}^\ell (\eta_i \circ u) (h \circ T u) = \sum_{i=1}^\ell (\eta_i \circ u) \bigl((h \circ T\varphi_i^*) \circ T(\varphi_i \circ u)\bigr)
\]
almost everywhere in \(M\). By lemma~\ref{likeGN}, every term in the right sum is weakly differentiable and so \(h\circ\,Tu\) is weakly differentiable. 

%----
Next, we prove that~\eqref{compactMorphism} implies~\eqref{twice}. 
First, let \((U_i)_{i\in I}\) be an open cover of \(N\) by sets given by lemma~\ref{lemmeExtended}. 
Let \(h \in C^1_c(T^*M\otimes TN, \R)\). 
Since the set \(\supp (h)\) is compact, in particular, there exist \(\ell \in \N_*\) and a finite partition of unity \((\eta_i)_{1\le i \le \ell}\)\ subordinate to \((U_i)_{1\le i \le \ell}\) such that almost everywhere in \(M\)
\begin{equation*}
h \circ\,Tu = \sum_{i=1}^\ell (\eta_i \circ u) (h\circ\,Tu) = \sum_{i=1}^\ell (\eta_i \circ u) \bigl(h \circ T\varphi_i^* \circ (T\varphi_i \circ\,Tu)\bigr).  
\end{equation*}
By assumption \(T \varphi_i \circ\,Tu\) is weakly differentiable, and thus by the chain rule for weakly differentiable functions (see for example~\citelist{\cite{willem}*{theorem 6.1.13}\cite{eg}*{theorem 4.2.4 (ii)}}), each term in the right sum is weakly differentiable, and so \(h\circ\,Tu \) is weakly differentiable. 
By proposition~\ref{propoTTU}, there exists a measurable map \(T^2 u \colon M \to \Mor(T^2M, T^2N)\) such that for every \(h \in C^1 (T^* M \otimes TN, \R)\) such that \(h\circ\,Tu \in W^{1,1}_\mathrm{loc}(M)\), we have almost everywhere in \(M\)
\[
 T (h \circ T u) = Th \circ T^2 u.
\]

We now prove that the integrability condition is satisfied. 
Let \((U_i)_{i\in I}\) be an open cover of \(N\) by sets given by lemma~\ref{lemmeExtended}. 
Let \(K\subseteq M\) and \(L\subseteq N\) be two compact subsets. 
Since the set \(L\) is compact, there exists \(\ell \in \N_*\) such that \(L \subseteq \bigcup_{i=1}^\ell U_i\). 
By construction, for every \(i \in \{1,\dots,\ell\}\), we have \(T^2(\varphi_i \circ u)= T^2 \varphi_i \circ T^2 u \) almost everywhere in \(M\) and, by lemma~\ref{lemmeExtended}, \(T^2\varphi_i^* \circ T^2 \varphi_i = \mathrm{id}\) on \((\pi_N^2)^{-1}(U_i)\), and so 
\(
T^2 u = T^2 \varphi^*_i \circ T^2(\varphi_i \circ u) 
\) 
almost everywhere in \(u^{-1}(U_i)\). 
Since for every \(i\in \{1,\dots,\ell\}\) we have \(\varphi_i\in C^2_c(N,\R^n)\), by proposition~\ref{compoDoubleVectorMorphism}, there exist constants \(C_1,\dots,C_\ell >0\) such that almost everywhere in \(u^{-1}(L) \cap K\)
\begin{align*}
\double{T^2 u}
& \le \sum_{i=1}^\ell \, \double{T^2\varphi^*_i \circ T^2(\varphi_i \circ u)}  
\le \sum_{i=1}^\ell C_i \, \double{T^2 \varphi^*_i} \, \double{T^2(\varphi_i \circ u)} \\
& \le (\max_{1\le i \le \ell} C_i) \bigl(\max_{1\le i \le \ell} \sup_{\R^n} \, \double{T^2 \varphi^*_i}  \bigr) 
\sum_{i=1}^\ell \, \double{T^2(\varphi_i \circ u)} .
\end{align*}
If \((\eta_j)_{j\in J}\) is a partition of unity subordinate to an atlas \(((\psi_j,V_j))_{j \in J}\) of local charts of \(M\), by the maximality property of double norms, there exists a positive continuous function \(\beta \in C(M\times \R^n,\R)\) such that for every \(i \in \{1,\dots,\ell\}\), almost everywhere in \(M\) 
\begin{multline*}
\double{T^2(\varphi_i \circ u)}
\le \beta(\cdot, \varphi_i \circ u) \sum_{j\in J} \eta_j 
\bigl( \abs{D(\varphi_i \circ u \circ \psi^{-1}_j) \circ \psi_j }^2_{g^*_m \otimes g_n} 
+ \abs{D(\varphi_i \circ u \circ \psi^{-1}_j) \circ \psi_j }_{g^*_m \otimes g_n} \\
+ \abs{D^2(\varphi_i \circ u \circ \psi^{-1}_j) \circ \psi_j }_{g^*_m \otimes g^*_m \otimes g_n}\bigr),
\end{multline*}
where \(g_m\) and \(g_n\) are the Euclidean metrics on \(\R^m\) and \(\R^n\) respectively. 
Since for every \(i\in \{1,\dots,\ell\}\), \(\varphi_i \circ u \in W^{2,1}_\mathrm{loc}(M,\R^n) \cap L^\infty(M,\R^n)\), by the Gagliardo--Nirenberg inequality~\citelist{\cite{gagliardo}\cite{nirenberg}}, \(\double{T^2(\varphi_i \circ u)} \in L^1(u^{-1}(L) \cap K)\) and so 
\[
\int_{u^{-1}(L) \cap K} \double{T^2 u} < + \infty. 
\]

Finally, we prove that~\eqref{twice} implies~\eqref{chainRule}. 
We take a function \(\theta \in C^1_c([0,+\infty),\R)\) such that \(0 \le \theta \le 1\) and \(\theta = 1\) on \([0,1]\). 
Then for every \(\ell\in \N\), we define the map \(\theta_\ell \colon T^*M\otimes TN \to \R\) for every \(y \in T^*M\otimes TN\) by 
\[
\theta_\ell (y) = \theta\Bigl(\frac{\abs{y}_{g_{T^*M\otimes TN}}}{1+\ell}\Bigr),
\]
where \(\abs{\cdot}_{g_{T^*M\otimes TN}}\) is a norm induced by a Riemannian metric on \(T^*M\otimes TN\)~\cite{docarmo}*{\S 1.2 proposition 2.10}.
Let \(f\in C^2_c(N,\R)\) and let \(K\subseteq M\) be a compact set. 
For every \(\ell \in \N\), we define \(\Tilde{f}_\ell = \theta_\ell \cdot (Tf)\). 
Since the map \(\Tilde{f}_\ell \colon T^*M\otimes TN \to \R\) is Lipschitz-continuous and its support is compact, \(\Tilde{f}_\ell \circ\,Tu\) is weakly differentiable~\cite{cvs}*{proposition 2.2} and there exist constants \(\Cl{a},\Cl{b}, \Cl{c} > 0\) such that for every \(\ell \in \N\) and almost everywhere in \(u^{-1}(\supp(f)) \cap K\)
\begin{align*}
& \abs{T(\Tilde{f}_\ell \circ T u)}_{\mathcal{L}} \\
& \hspace{5mm} \le \Cr{a} \,\Bigabs{\theta'\Bigl(\tfrac{\abs{Tu}_{g_{T^*M\otimes TN}}}{1+\ell}\Bigr)} \tfrac{1}{1+\ell}\, \double{T^2 u}\, \abs{T(f \circ u)}_{\mathcal{L}}
+ \Cr{b}\, \theta\Bigl(\tfrac{\abs{Tu}_{g_{T^*M\otimes TN}}}{1+\ell}\Bigr)\double{T^2 f}\, \double{T^2u} \\
& \hspace{5mm} \le \Cr{c} \,\double{T^2u}.
\end{align*}
Since the sequence \((\abs{T(\Tilde{f}_\ell \circ\,Tu)}_{\mathcal{L}})_{\ell \in \N}\) is bounded and uniformly integrable and since \((\Tilde{f}_\ell \circ T u)_{\ell \in \N}\) converges almost everywhere to \(Tf \circ\,Tu\) in \(K\subseteq M\), in view of the weak compactness criterion in \(L^1(K)\)~\citelist{\cite{bogachev}*{corollary 4.7.19}\cite{brezis}*{theorem 4.30}}, the sequence \((T(\Tilde{f}_\ell\circ\,Tu))_{\ell \in\mathbb{N}}\) converges weakly to \(T^2(f\circ u)\) in \(L^1(K)\) and \(T(f\circ u)\) is weakly differentiable. Hence, the map \(f\circ u\) is twice weakly differentiable.
\end{proof}

%*****************************************
\subsection{Sequences of maps having the chain rule property}
Since the property \(f\circ u\) is twice weakly differentiable for every \(f\in C^2_c(N,\R)\) is stronger than double differentiability, one can wonder whether using this stronger property in the definition of Sobolev maps might not lead to better spaces.

Although there is a closure property for such maps in terms of double norms, the next example shows that there is no closure property in terms of covariant derivatives, which are the quantities that can be controlled and observed in Riemmanian geometry.

\begin{example}\label{conditionSecondCovariant}
Let \(m \in \N\) with \(m \ge 3\) and let \(\alpha > 0\).
Let \((u_\ell)_{\ell\in\N}\) be a sequence such that for every \(\ell \in \N\), the map \(u_\ell \colon \R^m \to \mathbb{S}^1 \subseteq \R^2\) is defined by 
\(u_\ell = (\cos \circ \,v_\ell, \sin \circ \,v_\ell)\),
where the function \(v_\ell \colon \R^m \to \R\) is defined for every \(x\in \R^m\) by 
\[
v_\ell(x) =\Bigl(\abs{x} + \frac{1}{1+\ell} \Bigr)^{-\alpha}.
\]
For every \(\ell\in \N\), \(\abs{Du_\ell}_{g^*_m\otimes g_{\mathbb{S}^1}} = \abs{D v_\ell}_{g^*_m\otimes g_1}\) and \(\abs{D_K( T u_\ell)}_{g^*_m \otimes g^*_m\otimes g_{\mathbb{S}^1}} = \abs{D^2 v_\ell}_{g^*_m\otimes g^*_m \otimes g_1}\). 
The sequence \((v_\ell)_{\ell\in \N}\) converges locally in measure to the map \(v\colon \mathbb{B}^m_r \to \R\) defined for every \(x\in \mathbb{B}^m_r \setminus\{0\}\) by \(v(x)=\abs{x}^{-\alpha}\). 

If \(\alpha < m - 2\), then the sequences \((Du_\ell)_{\ell \in \N}\) and \((D_K(T u_\ell))_{\ell \in \N}\) are bilocally uniformly integrable. 
Every term of the sequence \((u_\ell)_{\ell\in \N}\) satisfies the second-order chain rule property and the sequence \((u_\ell)_{\ell \in \N}\) converges locally in measure to the map \(u=(\cos \circ \,v,\sin \circ \,v) \colon \Bset^m_r \to \mathbb{S}^1\). 
However, if \(\alpha > \tfrac{m-2}{2}\), the map \(u\) does not satisfy the second-order chain rule. 
Indeed, if we take \(f \in C^2(\mathbb{S}^1, \R^2)\) such that \(f=\mathrm{id}\) on \(\mathbb{S}^1\subseteq \R^2\), then for every \(x\in \R^m\setminus \{0\}\), we have 
\(\abs{D^2(f \circ u)(x)}_{g^*_m \otimes g^*_m\otimes g_2} \ge \alpha \abs{x}^{-2\alpha -2}\) and so \(\abs{D^2(f \circ u)}_{g^*_m \otimes g^*_m\otimes g_2}\) does not belong to \(L^1_\mathrm{loc}(\Bset^m_r)\). 
\end{example}

In other words, the previous example exhibits a sequence \((u_\ell)_{\ell\in \N}\) satisfying the same assumptions on the sequences \((Tu_\ell)_{\ell\in \N}\) and \((D_K (Tu_\ell))_{\ell\in \N}\) together with a second-order chain rule than those of proposition~\ref{closureProperty} but with a limit that does not have the second-order chain rule property. 

In fact, given a sequence of measurable maps \((u_\ell)_{\ell\in \N}\) such that for every \(f\in C^2_c(N,\R)\), \(f\circ u_\ell\) is twice weakly differentiable, the sequence \((\double{T^2 u_\ell})_{\ell\in \N}\) has to be uniformly integrable.

%*********************************************
\subsection{Chain rule for higher order colocally weakly differentiable maps}\label{higherOrderChainRule}

Given a \(k\) times colocally weakly differentiable map \(u\colon M \to N\), we can also investigate whether for every \(f\in C^k_c(N,\R)\), \(f\circ u\) is \(k\) times weakly differentiable. 

In a first step, the notion of double norm can be generalized to a notion of \(k^\textrm{tuple}\) norm on \(k^{\text{th}}\) order tangent bundles. 
A map \(\double{\cdot} \colon T^k M \to \R\) is a \emph{\(k^\textrm{tuple}\) seminorm} on \(T^k M\) if 
\begin{enumerate}[(i)]
\item the map \(x\in M \mapsto \double{\cdot}_{\arrowvert T_x^k M}\) is continuous, 
\item \(\double{\cdot}\) is a seminorm on fibers over each of the \(k\) vector bundle structure. 
\end{enumerate}
For example, if \(M = \R^m\) and \(I_k = 2^{\{1,\dots, k\}} \setminus \varnothing\), then the \emph{canonical \(k^\textrm{tuple}\) seminorm} is defined for every \(\nu = (x, (e_\lambda)_{\lambda \in I_k}) \in T^k M\) by 
\[
\double{\nu}_* 
= \sum_{\ell =1}^k \sum_{\lambda\in \mathcal{P}_\ell (\{1,\dots,k\})} \abs{e_{\lambda^1}} \dotsm \abs{e_{\lambda^\ell}},
\] 
where \(\mathcal{P}_\ell(\{1,\dots,k\}) = \{ \lambda = \{\lambda^1,\dots, \lambda^\ell\} \in 2^{I_k} \colon \dot{\cup}_{i=1}^\ell \lambda^i = \{1,\dots,k\}, \varnothing \not\in \lambda\}\) is the set of all partitions of \(\{1,\dots,k\}\) of length \(\ell\). 
As for the canonical double seminorm, for every \(k^\textrm{tuple}\) seminorm \(\double{\cdot}\) on \(T^k M\), there exists \(C >0\) such that for every \(\nu \in T^k M\), \(\double{\nu} \le C \double{\nu}_*\). 
As a consequence, for any manifold \(M\), there exists a \(k^\textrm{tuple}\) norm on \(T^k M\) defined as follows. 

\begin{de}\label{defkTupleNorm}
A map \(\double{\cdot} \colon T^k M \to \R\) is a \emph{\(k^\textrm{tuple}\) norm} on \(T^k M\) whenever 
\begin{itemize}
\item[(i)] \(\double{\cdot}\) is a seminorm on \(T^k M\), 
\item[(ii)] \(\double{\cdot}\) is maximal: for every \(k^\textrm{tuple}\) seminorm \(\double{\cdot}'\) on \(T^k M\), there exists a positive continuous function \(\beta \in C (M,\R)\) such that for every \(\nu \in T^k M\), 
\[
\double{\nu}' \le \beta (\pi^k_M (\nu)) \double{\nu},
\]
where \(\pi^k_M \colon T^k M \to M\) is the canonical submersion. 
\end{itemize}
\end{de}

\begin{rem}
Unlike the double norm, the \(k^\textrm{tuple}\) norm is not a norm on fibers over non zero elements.
For example, if \(M=\R^m\), for every \(\nu = (x,e_1,e_2,e_3, e_{12}, e_{23}, e_{13}, e_{123}) \in T^3 M\),
\[
\double{\nu}_* = 
\abs{e_1}\,\abs{e_2}\,\abs{e_3} + \abs{e_1}\,\abs{e_{23}} + \abs{e_2}\,\abs{e_{13}} + \abs{e_3}\,\abs{e_{12}} + \abs{e_{123}}. 
\]
In particular, if \(\nu =(x,0,0,0,e,e,e,0)\), then \(\pi_{T^2M}(\nu) = T\pi_{TM}(\nu) = T^2\pi_M(\nu)\) and
\(\double{\nu}_* = 0\) but \(\nu \neq 0_{\pi^3_M(\nu)}\) whenever \(e\neq 0\). 
\end{rem}

Since \(\Mor(T^k M, T^k N)\) has a \(k^\text{tuple}\) vector bundle structure, we can also define the notion of \(k^\textrm{tuple}\) norm on this space. 
For example, if \(M =\R^m\) and \(N = \R^n\), a \emph{canonical \(k^\textrm{tuple}\) norm} \(\double{\cdot}_*\) on \(\Mor(T^k M, T^k N)\) can be defined for every \(f\in \Mor(T^k M,T^k N)\) by 
\[
\double{f}_* = \sum_{\ell = 1}^k \sum_{\Lambda \in \mathcal{P}_\ell(\{1,\dots,k\})} 
\left(\sum_{j = 1}^{\ell} \sum_{\lambda \in \mathcal{P}_j(\Lambda)} \abs{f_{\lambda^1}} \dotsm \abs{f_{\lambda^j}} \right)
\]
since for every \(\nu\in T^k M\), 
\[
f(\nu) =\left( y, \left(\sum_{i=1}^{\abs{\Lambda}} \sum_{\lambda \in \mathcal{P}_i(\Lambda)} f_{\lambda}(e_{\lambda^1},\dotsc,e_{\lambda^i}) \right)_{\Lambda\in I_k} \right),
\]
where for every \(\Lambda\in I_k\), for every \(1\le i \le \abs{\Lambda}\) and every \(\lambda \in \mathcal{P}_i(\Lambda)\), the map \(f_\lambda \colon \times^i \R^m \to \R^n\) is \(i\)-linear. 

In a second step, we state equivalent assertions to the chain rule for higher-order colocally weakly differentiable maps. 
We recall that a locally integrable map \(u\colon M \to \R\) is \emph{\(k\) times weakly differentiable} if for every \(x\in M\), there exists a local chart \(\psi \colon V \subseteq M \to \R^m\) such that \(x\in V\) and the map \(u\circ \psi^{-1}\) belongs to \(W^{k,1}_\mathrm{loc}(\psi(V))\). 
For every map \(h\colon \Mor(T^k M,T^k N) \to\R^q \), we denote by \(\supp_N(h)\) the \emph{support of \(h\) with respect to \(N\)}, that is, 
\(
\supp_N( h) = \pi_N (\pi^k_{M\times N} (\supp(h))), 
\) 
where \(\pi^k_{M\times N} \colon \Mor(T^k M, T^k N) \to M \times N\) is a natural submersion. 

\begin{propo}\label{kChainRule}
Let \(j\in \{2,\dots,k\}\). Let \(u \colon M \to N\) be a \(j\) times colocally weakly differentiable map. 
The following statements are equivalent. 
\begin{enumerate}[(i)]
\item\label{kTimesChainRule} 
For every \(f\in C^k_c(N,\R)\), \(f\circ u\) is \(k\) times weakly differentiable, 
\item\label{kCompactMorphism}
for every \(f\in C^j_c(N,\R)\), \(f\circ u\) is \(j\) times weakly differentiable and for every \(h\in \Mor(\Mor(T^j M, T^j N),\R)\) of class \(C^{k-j}\) such that \(\supp_N (h)\) is compact, the map \(h\circ T^j u\) is \((k-j)\) times weakly differentiable,
\item\label{kTimes} 
the map \(u\) is \(k\) times colocally weakly differentiable and for all compact subsets \(K\subseteq M\) and \(L\subseteq N\), for all \(\ell\in \{2,\dots,k\}\), for all \(\ell^\text{tuple}\) norm \(\double{\cdot}\) on \(\Mor(T^\ell M, T^\ell N)\), 
\begin{equation*}
\int_{u^{-1}(L) \cap K} \double{T^\ell u} < + \infty. 
\end{equation*}
\end{enumerate}
Moreover, if for every \(f\in C^k_c(N,\R)\), \(f\circ u\) is \(k\) times weakly differentiable, then for all compact subsets \(K\subseteq M\) and \(L\subseteq N\), for all \(j^\text{tuple}\) norm \(\double{\cdot}\) on \(\Mor(T^j M, T^j N)\), 
\begin{equation*}
\int_{u^{-1}(L) \cap K} \double{T^j u}^\frac{k}{j} < + \infty.
\end{equation*}
\end{propo}

To prove this proposition, we need similar intermediate results to those of proposition~\ref{equivalentChainRule}. In fact, we can state every intermediate results (lemmas~\ref{likeGN} and~\ref{lemmeExtended}) with all higher-order notions and then the proof of proposition~\ref{kChainRule} has the same structure. 

If \(M\) is an open set of \(\R^m\) and \(N = \R\), it is well known that the left composition does not operate on higher-order Sobolev spaces~\citelist{\cite{adams}\cite{bourdaud}*{th\'eor\`eme 3 (i)}\cite{dahlberg}}. 
Thanks to proposition~\ref{kChainRule}, we have some examples for which the chain rule does not hold by giving maps that belong to \(W^{k,1}_\mathrm{loc}(\R^m)\) but that do not satisfy the integrability condition in assertion~\eqref{kTimes}. 
\begin{example}
Let \(\beta > \alpha >0\). The function \(u \colon \R^m \to \R\) is defined for every \(x\in \R^m\) by 
\[
u(x) = 
\begin{cases}
\frac{1}{\abs{x}^\alpha}\sin(\abs{x}^\beta) & \text{if } x \neq 0, \\
0 & \text{if } x= 0. 
\end{cases}
\]
If \(0 < \alpha < m \) and \(\frac{1}{3}(\alpha - m + 3) < \beta < \frac{1}{3} (3 \alpha - m +3)\), the map \(u\) belongs to \(W^{3,1}_\mathrm{loc}(\R^m)\) but there exist \(\gamma > 0\) and \(\delta >0\) such that for almost every \(x \in u^{-1}([-1,1]) \cap \overline{\mathbb{B}}^m_\delta\), 
\[
\abs{Du(x)}_{g^*_m\otimes g_1}\ge \gamma (\beta - \alpha) \abs{x}^{\beta - \alpha -1}. 
\]
Since there exists \(C >0\) such that \(\double{T^3 u} \ge C \abs{Du}^3_{g^*_m\otimes g_1}\) and since the map \(x \in \R^m\setminus \{0\} \mapsto \abs{x}^{\beta - \alpha -1}\) does not belong to \(L^3(u^{-1}([-1,1]) \cap \overline{\mathbb{B}}^m_\delta)\), the map \(u\) does not satisfy the integrability condition in assertion~\eqref{kTimes}.  
\end{example}

%*********************************************
\section{Back to the definition by embedding}

%*********************************************
\subsection{Comparison with the intrinsic definition} 
Given an isometric embedding of the target manifold, we compare the notion of higher-order Sobolev spaces by embedding~\eqref{Wkp} and the intrinsic one (definitions~\ref{defW2p} and~\ref{defWkp}). 

Let \(k\ge 2\). 
We assume that \((M,g_M)\) and \((N,g_N)\) are Riemannian manifolds with the respective Levi--Civita connection maps.  
For every \(q \ge 1\), we denote by \(g_q\) the Euclidean metric on \(\R^q\). 

If the target manifold is compact, we characterize higher-order Sobolev spaces by embedding~\eqref{Wkp}. 
A measurable map \(u\colon M \to N\) belongs to \(L^\infty(M,N)\) if 
\[
\mathrm{osc}(u) = \operatorname{ess\,sup}\, \bigl\{ d_N(u(x), u(y)) \colon x,y \in M \bigr\} < +\infty,
\]
where \(d_N\) is the geodesic distance on \(N\) induced by the Riemannian metric \(g_N\), and for every \(p \in [1,+\infty)\), we say that \(M\) has the \emph{Gagliardo--Nirenberg property}~\citelist{\cite{gagliardo}\cite{nirenberg}}
if there exists a constant \(C >0\) such that for every \(v\in \dot{W}^{k,p}(M, \R) \cap L^\infty(M,\R)\) and for every \(j \in \{1,\dots,k-1\}\),
\begin{equation}\label{GN}
\int_M \abs{D^j_K v}^\frac{kp}{j}_{(\otimes^j g^*_M) \otimes g_1} 
\le C \, \mathrm{osc}(v)^{\frac{k-j}{j} p} \int_M \abs{D^k_K v}_{(\otimes^k g^*_M)\otimes g_1}^p. 
\end{equation}

\begin{propo}\label{iotaCompact}
Let \(\iota \in C^k(N, \R^\nu)\) be an isometric embedding and let \(p \in [1,+\infty)\). 
If \(M\) has the Gagliardo--Nirenberg property and if \(N\) is compact, then 
\[
\dot{W}^{k,p}_\iota (M,N) = \bigcap_{j=1}^k \dot{W}^{j,\frac{kp}{j}}(M,N). 
\]
\end{propo}

Without any assumptions on the manifolds, we begin by proving a lemma that concerns the notion of \(k\) times colocally weakly differentiable maps. 
\begin{lemme}\label{iotaKTimes}
Let \(\iota \in C^k(N, \R^\nu)\) be an isometric embedding and let \(u \colon M \to N\) be a measurable map. 
If the map \(\iota \circ u\) is \(k\) times weakly differentiable, then the map \(u\) is \(k\) times colocally weakly differentiable. 
\end{lemme}

In contrast with proposition~\ref{equivalentChainRule}, the hypotheses involve only a single embedding of class \(C^k\). 

\begin{proof}[Proof of lemma~\ref{iotaKTimes}]
If the map \(\iota \circ u\) is weakly differentiable, then \(u\) is colocally weakly differentiable and 
\(T(\iota \circ u) = T\iota \circ T u\) almost everywhere in \(M\)~\cite{cvs}*{proposition 1.9}. 
We then proceed by iteration. Let \(j\in \{1,\dots,k-1\}\) and let \(f\in C^1_c(\Mor(T^j M, T^j N),\R)\). 
Since \(\iota\) is an embedding, \(T^j \iota \colon T^j N \to T^j \R^\nu\) is also an embedding. 
Hence, \(T^j \iota(T^j N)\) has a tubular neighborhood in \(T^j \R^\nu\): there exists a vector bundle \((E,\pi_{T^j N}, T^j N)\) and an embedding \(\Tilde{\iota} : E \to T^j \R^\nu\) such that \(\Tilde{\iota}_{\arrowvert T^j N} = T^j \iota\) and \(\Tilde{\iota}(E)\) is open in \(T^j \R^\nu\)~\cite{hirsch}*{theorem 4.5.2}. 
Thanks to the tubular neighborhood, there exists a map \(\Tilde{f} \in C^1_c(\Mor(T^j M, T^j \R^\nu),\R)\) such that \(\Tilde{f} \circ T^j \iota = f\) on \(\Mor(T^j M,T^j N)\) (see for example proof of proposition 1.9 in~\cite{cvs}). In particular, almost everywhere in \(M\)
\[
f \circ T^j u = \Tilde{f} \circ T^j \iota \circ T^j u = \Tilde{f} \circ T^j(\iota \circ u).
\]
Since \(T^j(\iota \circ u)\) is weakly differentiable, the map \(T^j(\iota \circ u)\) is colocally weakly differentiable and so \(f \circ T^j u\) is weakly differentiable. Finally, since \(T^j \iota \circ T^j u\) is weakly differentiable, \(T^{j+1}(\iota \circ u) = T^{j+1} \iota \circ T^{j+1} u\) almost everywhere in \(M\)~\cite{cvs}*{proposition 1.5}. 
\end{proof}

\begin{proof}[Proof of proposition~\ref{iotaCompact}]
\resetconstant
On the one hand, let \(u \in \dot{W}^{k,p}_\iota(M,N)\). 
By lemma~\ref{iotaKTimes}, the map \(u\) is \(k\) times colocally weakly differentiable. 
For every \(j\in \{2,\dots,k\}\), since \(T^j(\iota \circ u) = T^j \iota\circ T^j u\) almost everywhere in \(M\), there exists a constant \(\Cl{ixaye}>0\) such that 
\begin{align*}
\abs{D^j_K u}_{(\otimes^j g^*_M) \otimes g_N}
& = \abs{D\iota(u)[D^j_K u]}_{(\otimes^j g^*_M) \otimes g_\nu} \\
& \le \abs{D^j_K(\iota \circ u)}_{(\otimes^j g^*_M) \otimes g_\nu} \\
& \hspace{10mm} + \Cr{ixaye} \sum_{\ell=2}^{j} \sum_{i=1}^{j-l+1}\abs{D^\ell_K \iota(u)}_{(\otimes^\ell g^*_N) \otimes g_\nu}  
\abs{D^i_K u}^\frac{j}{i}_{(\otimes^i g^*_M) \otimes g_N}
\end{align*}
almost everywhere in \(M\). Since \(\iota \circ u \in \dot{W}^{k,p}(M,\R^\nu) \cap L^\infty(M,\R^\nu)\), by the Gagliardo--Nirenberg property~\eqref{GN}, 
for every \(j \in \{1,\dots,k\}\), \(\abs{D^j_K(\iota \circ u)}_{(\otimes^j g^*_M) \otimes g_\nu} \in L^\frac{kp}{j}(M)\). 
As a consequence, by iteration and the previous inequality, \(u\in \bigcap_{j=1}^k \dot{W}^{j,\frac{kp}{j}}(M,N)\). 

On the other hand, let \(u \in \bigcap_{j=1}^k \dot{W}^{j,\frac{kp}{j}}(M,N)\). 
By definition~\ref{defkTupleNorm} of \(k^\textrm{tuple}\) norms and by Young's inequality, for every \(j \in \{2,\dots,k\}\) and every compact subset \(K \subseteq M\), there exists a constant \(\Cl{beta}>0\) such that 
\[
\double{T^j u} \le \Cr{beta} \Bigl( \sum_{\ell=1}^k \abs{D^\ell_K u}_{(\otimes^\ell g^*_M) \otimes g_N}^\frac{k}{\ell} +1 \Bigr)
\] 
almost everywhere in \(K\) and so \(\double{T^j u} \in L^1_\mathrm{loc}(M)\). Since the manifold \(N\) is compact by assumption, by the chain rule for higher-order colocally weakly differentiable maps (proposition~\ref{kChainRule}), the map \(\iota \circ u\) is \(k\) times weakly differentiable. 
Moreover, there exists a constant \(\Cl{pyejp} >0\) such that 
\[
\abs{D^k_K(\iota \circ u)}_{(\otimes^k g^*_M) \otimes g_\nu} \le 
\Cr{pyejp} \sum_{\ell=1}^{k} \sum_{i=1}^{k-l+1}\abs{D^\ell_K \iota(u)}_{(\otimes^\ell g^*_N) \otimes g_\nu}  
\abs{D^i_K u}^\frac{k}{i}_{(\otimes^i g^*_M) \otimes g_N}
\]
almost everywhere in \(M\) and since \(N\) is compact, \(\iota \circ u \in \dot{W}^{k,p}(M,\R^\nu)\). 
\end{proof} 

Without the compactness and Gagliardo--Nirenberg assumption, proposition~\ref{iotaCompact} is not true. Indeed, the embedded space in the intersection already fails for classical Sobolev maps between Euclidean spaces. 
However, there exists an embedding such one inclusion always occurs.

\begin{propo}\label{particularIotaGeneral}
Let \(p\in [1,+\infty)\). If either \(k=2\) or \(M\) has the Gagliardo--Nirenberg property, then there exists an isometric embedding \(\iota \in C^k(N, \R^\nu)\) such that
\[
\dot{W}^{k,p}_\iota(M,N) \subseteq \bigcap_{j=1}^k \dot{W}^{j,\frac{k p}{j}}(M,N). 
\]
\end{propo}

\begin{proof}
\resetconstant
If \(M\) has the Gagliardo--Nirenberg property, for every \(u \in \dot{W}^{k,p}_\iota(M,N)\), by lemma~\ref{iotaKTimes}, the map \(u\) is \(k\) times colocally weakly differentiable. 
If \(\iota \in C^k(N,\R^\nu)\) is an isometric embedding such that \(\iota(N)\) is bounded, then \(\iota \circ u \in \dot{W}^{k,p}(M,\R^\nu)\cap L^\infty(M,\R^\nu)\) and so \(u \in \bigcap_{j=1}^k \dot{W}^\frac{k p}{j}(M,N)\) (as in the first part of the proof of proposition~\ref{iotaCompact}). 
For example, by the Nash embedding theorem~\citelist{\cite{nash54}\cite{nash56}}, such an embedding \(\iota\) always exists. 

If \(k=2\), let \(\iota_1 \in C^2(N, \R^\nu)\) be an isometric embedding given by Nash embedding theorem~\citelist{\cite{nash54}\cite{nash56}}. 
We define an isometric embedding \(\iota_2 \colon \R^\nu \to \R^{3 \nu}\) for every \(t\in \R^\nu\) by 
\[
\iota_2(t) = \left(\lambda t_1 , \gamma \cos\left(\mu t_1\right), \gamma \sin\left( \mu t_1 \right), \dotsc, \lambda t_\nu , \gamma \cos\left(\mu t_\nu\right), \gamma \sin\left( \mu t_\nu \right)\right)
\]
with \(\lambda, \gamma, \mu \in \R\) such that \(\lambda^2 + \gamma^2 \mu^2 =1\). 
Then \(\iota = \iota_2 \circ \iota_1 \colon N \to \R^{3\nu}\) is an isometric embedding and the second fundamental form \(A_\iota \colon TN \times TN \to \R^{3\nu}\) of \(\iota\)~\cite{docarmo}*{\S 6.2} satisfies for every \(v_1,v_2\in TN\), 
\begin{align*}
\abs{A_\iota [v_1,v_2]}^2 
& = \abs{A_{\iota_1}[v_1,v_2]}^2 + \abs{A_{\iota_2}[D\iota_1(v_1),D\iota_1(v_2)]}^2 \\
& \ge \abs{A_{\iota_2}[D\iota_1(v_1),D\iota_1(v_2)]}^2 
= \gamma^2 \mu^4 \sum_{i=1}^\nu (\psh{D\iota_1(v_1)}{e_i})^2 (\psh{D\iota_1(v_2)}{e_i})^2,
\end{align*}
where \((e_i)_{1\le i\le\nu}\) is the canonical basis of \(\R^\nu\).
So there exists \(\Cl{mnxjb} >0\) such that for every \(v\in TN\), 
\[
\abs{v}_{g_N}^2 = \abs{D\iota_1(v)}_{g_\nu}^2 \le \Cr{mnxjb} \abs{A_\iota [v,v]}_{g_{3\nu}}. 
\]

Finally, for every \(u \in \dot{W}^{2,p}_\iota(M,N)\), since \(\iota\) is an isometric embedding, by orthogonal decomposition of \(T\R^{3\nu}\) into \(TN\) and its orthogonal complement in \(T\R^{3\nu}\)~\cite{docarmo}*{\S 6.2}, for almost every \(x\in M\) and every \(e_1,e_2 \in T_x M\),
\[ 
\abs{D^2_K (\iota \circ u)(x)[e_1,e_2]}_{g_{3\nu}}^2 
= \abs{D_K^2 u(x)[e_1,e_2]}_{g_N}^2 + \abs{A_{\iota}[Tu(x)[e_1], Tu(x)[e_2]]}_{g_{3\nu}}^2.
\]
Consequently, \(u \in \dot{W}^{2,p}(M,N)\) and since for almost every \(x\in M\),
\begin{align*}
\abs{D^2_K(\iota \circ u)(x)}_{g_M^* \otimes g^*_M \otimes g_{3\nu}}
& \ge \sum_{1\le i,j\le m} \abs{A_\iota[Tu(x)[e_i], Tu(x)[e_j]]}_{g_{3\nu}} \\ 
& \ge \sum_{1\le i\le m} \abs{A_\iota[Tu(x)[e_i], Tu(x)[e_i]]}_{g_{3\nu}} \\ 
& \ge  \sum_{1\le i\le m} \frac{1}{\Cr{mnxjb}} \abs{Tu(x)[e_i]}_{g_N}^2
 = \frac{1}{\Cr{mnxjb}} \abs{Tu(x)}_{g^*_M \otimes g_N}^2, 
\end{align*}
where \((e_i)_{1\le i \le m}\) is an orthonormal basis in \(\pi_M^{-1}(\{x\})\), \(u \in \dot{W}^{1,2p}(M,N)\). 
\end{proof}

It turns out that the definition by embedding~\eqref{Wkp} and the intrinsic one may be different. 
In general, we do not even know if one or another inclusion does occur, except in the particular case of \(k=2\). 

\begin{propo}\label{inclusionEmbedding}
Let \(\iota \in C^2(N, \R^\nu)\) be an isometric embedding and let \(p \in [1,+\infty)\).
Then 
\[
\dot{W}^{2,p}_\iota (M,N) \subseteq \dot{W}^{2,p}(M,N).
\]
\end{propo}

\begin{proof}
For every \(u \in \dot{W}^{2,p}_\iota (M,N)\), by lemma~\ref{iotaKTimes}, \(u\) is twice colocally weakly differentiable. 
Since \(\iota\) is an isometric embedding, by orthogonal decomposition of \(T\R^\nu\) into \(TN\) and its orthogonal complement in \(T\R^\nu\)~\cite{docarmo}*{\S 6.2}, 
\(\abs{D_K^2 u}_{g^*_M\otimes g^*_M \otimes g_N} \le \abs{D^2_K(\iota \circ u)}_{g^*_M\otimes g^*_M \otimes g_\nu}\) almost everywhere in \(M\), and so \(u \in \dot{W}^{2,p}(M,N)\). 
\end{proof}

%******************************************
\subsection{Gagliardo--Nirenberg property}

In view of proposition~\ref{iotaCompact}, if the target manifold is compact, we may ask if there exist some Gagliardo--Nirenberg inequalities~\citelist{\cite{gagliardo}\cite{nirenberg}} for the spaces \(\dot{W}^{k,p}(M,N)\) that can lead to 
\[
\dot{W}^{k,p}_\iota(M,N) = \dot{W}^{k,p}(M,N). 
\] 
In general, even if the target manifold is compact, there are no such inequalities. 

A first striking fact is that the second-order energy can vanish for a nontrivial map.
\begin{propo}\label{sobolevNoGN}
Let \(p\in [1,+\infty)\). 
If \(N\) has at least one nontrivial closed geodesic, then there exists \(u \in \dot{W}^{k,p}(\mathbb{S}^1,N) \cap L^\infty(\mathbb{S}^1,N)\) such that \(\abs{D_K^k u}_{(\otimes^k g^*_{\mathbb{S}^1}) \otimes g_N} = 0\) almost everywhere in \(M\) but \(\int_{\mathbb{S}^1} \abs{Tu}^{k p}_{g^*_{\mathbb{S}^1} \otimes g_N} > 0\). 
\end{propo}

In particular, since a compact manifold has always at least one nontrivial closed geodesic~\cite{klingenberg}, if \(N\) is compact, there are no Gagliardo--Nirenberg inequalities~\citelist{\cite{gagliardo}\cite{nirenberg}} for the spaces \(\dot{W}^{k,p}(\mathbb{S}^1,N)\). 

\begin{proof}[Proof of proposition~\ref{sobolevNoGN}]
Let \(\ell \in \N_*\) and let \(u_\ell \colon \mathbb{S}^1 \to \mathbb{S}^1 \subseteq \mathbb{C}\) be defined for every \(\theta \in [0,2\pi]\) by 
\(u_\ell(e^{i \theta}) = e^{i \ell \theta}\).  
Let \(\gamma \colon \mathbb{S}^1 \to N\) be a nontrivial closed geodesic. 
Since \(\gamma\) is a geodesic, we assume that for all \(y\in \mathbb{S}^1\), \(\abs{T\gamma(y)}_{g_{\mathbb{S}^1}^*\otimes g_N} = 1\). 

Then \(u = \gamma \circ u_\ell\) is \(k\) times colocally weakly differentiable, it belongs to \(L^\infty(\mathbb{S}^1,N)\) and 
\(\abs{D_K^k u}_{(\otimes^k g^*_{\mathbb{S}^1}) \otimes g_N}  = 0\) almost everywhere in \(M\) 
but 
\[
\int_{\mathbb{S}^1} \abs{T u}_{g^*_{\mathbb{S}^1} \otimes g_N}^{k p} 
= \int_{\mathbb{S}^1} \abs{T u_\ell}_{g^*_{\mathbb{S}^1} \otimes g_{\mathbb{S}^1}}^{k p} = \ell^{k p} >0. \qedhere
\]
\end{proof}

\begin{propo}\label{sobolevNoGN2}
Let \(p\in [1,+\infty)\) so that \(1\le 2p < m\) and let \(r >0\). 
If \(N\) has at least one nontrivial bounded geodesic, then there exists \(u \in \dot{W}^{2,p}(\mathbb{B}^m_r,N) \cap L^\infty(\mathbb{B}^m_r,N)\) such that \(\int_{\mathbb{B}^m_r} \abs{Tu}_{g^*_m \otimes g_N}^{2p} = +\infty\). 
\end{propo}

If \(\iota \in C^2(N,\R^\nu)\) is an isometric embedding given by proposition~\ref{particularIotaGeneral}, it follows that 
\[
\dot{W}^{2,p}_\iota(\mathbb{B}^m_r,N) \subsetneq \dot{W}^{2,p}(\mathbb{B}^m_r,N), 
\]
and if \(N\) is compact, by propositions~\ref{iotaCompact} and~\ref{inclusionEmbedding}, the strict inclusion occurs for any embedding \(\iota\). 

\begin{proof}[Proof of proposition~\ref{sobolevNoGN2}]
Let \(\alpha > 0\). 
We define the function \(v \colon \mathbb{B}^m_r \setminus \{0\} \to \R\) for every \(x\in \mathbb{B}^m_r \setminus \{0\}\) by
\(v(x)= \abs{x}^{-\alpha}.\)
Let \(\gamma \colon \R \to N\) be a nontrivial bounded geodesic.
Since \(\gamma\) is a geodesic, we assume that for all \(y\in \R\), \(\abs{T\gamma(y)}_{g_{\mathbb{S}^1}^*\otimes g_N} = 1\). 

Then \(u = \gamma \circ v\) is twice colocally weakly differentiable, \(\abs{T u}_{g^*_m \otimes g_N} = \abs{Tv}_{g^*_m \otimes g_1}\) and since \(\gamma\) is a geodesic, almost everywhere in \(\mathbb{B}^m_r \setminus \{0\}\)
\[
\abs{D^2_K u}_{g^*_m \otimes g^*_m \otimes g_N} \le \abs{D^2 v}_{g^*_m \otimes g^*_m \otimes g_1}.
\]
If \((2+\alpha)p < m < 2 p(\alpha+1)\), then \(u \in \dot{W}^{2,p}(\mathbb{B}^m_r,N)\) but \(u \notin \dot{W}^{1,2p}(\mathbb{B}^m_r,N)\).
\end{proof}

For other manifolds \(M\) and \(N\) and \(p \in [1,+\infty)\), it leads to the following open question and the answer involves in particular the geometry of \(M\) and \(N\). 
\begin{open}\label{openQuestionGN}
If the manifolds \(M\) and \(N\), \(k\ge 2\) and \(p\in[1,+\infty)\) do no satisfy the hypotheses of propositions~\ref{sobolevNoGN} and~\ref{sobolevNoGN2}, does exist a constant \(C >0\) such that 
for every \(u\in \dot{W}^{k,p}(M,N) \cap L^\infty(M,N)\), for every \(j\in \{1,\dots,k-1\}\),
\begin{equation*}
\int_M \abs{D^j_K u}_{(\otimes^j g^*_M) \otimes g_N}^\frac{k p}{j} \le C \, \mathrm{osc}(u)^{\frac{k-j}{j}p} \int_M \abs{D_K^k u}_{(\otimes^k g^*_M) \otimes g_N}^p?
\end{equation*}
\end{open}

%***********************************************
\subsection{Problem of density of smooth maps}
In this part, we remark that intrinsic definitions~\ref{defW2p} and~\ref{defWkp} and proposition~\ref{iotaCompact} give rise to new open questions whether the manifolds \(M\) and \(N\) are compact.  

In view of proposition~\ref{iotaCompact}, given a map \(u\in \dot{W}^{k,p}(M,N)\), we may ask whether there exists a sequence \((u_\ell)_{\ell \in \N} \subseteq \dot{W}^{k,p}_\iota(M,N)\) that \emph{converges strongly} to \(u\) in \(\dot{W}^{k,p}(M,N)\)~\cite{cvs}*{definition 4.1}, that is, the sequence \((T^k u_\ell)_{\ell \in \N}\) converges to \(T^k u\) locally in measure and the sequence \((\abs{D_K^k u_\ell}_{(\otimes^k g^*_M) \otimes g_N})_{\ell \in \N}\) converges to \(\abs{D_K^k u}_{(\otimes^k g^*_M ) \otimes g_N}\) in \(L^p(M)\).
In view of the strong density in Sobolev spaces between manifolds (see for example~\citelist{\cite{bethuel}\cite{bpvs}\cite{hl}}), we may also ask whether there exists a map \(u \in \dot{W}^{k,p}_\iota(M,N)\) that cannot be approximated by smooth maps in \(\dot{W}^{k,p}_\iota (M,N)\) but that can be in \(\dot{W}^{k,p}(M,N)\). 

For instance, the hedgehog map \(u_h \colon \mathbb{B}^m \to \mathbb{S}^{m-1}\) defined for every \(x\in \mathbb{B}^m \setminus \{0\}\) by \(u_h(x) = \frac{x}{\abs{x}}\) belongs to \(W^{k,p}_\iota (\mathbb{B}^m, \mathbb{S}^{m-1})\) for \(kp < m\) but \(u_h\) cannot be strongly approximated by maps in \(C^\infty(\overline{\mathbb{B}}^m,\mathbb{S}^{m-1})\) for \(kp \ge m-1\) since the identity map on \(\mathbb{S}^{m-1}\) does not have a continuous extension to \(\mathbb{B}^m\) with values into \(\mathbb{S}^{m-1}\)~\citelist{\cite{bz}*{\S II}\cite{su}*{\S 4 example}}. 

For the notion of intrinsic higher-order Sobolev maps, the space is larger than the one by embedding (proposition~\ref{iotaCompact}) but the notion of convergence is weaker. However, even if the question of strong density is still open, the necessary and sufficient condition that appears with definition by embedding~\cite{bpvs}*{theorem 1} is necessary for the intrinsic one. 

We prove the results for the case \(k=2\) but we note that those results extend to higher-order. 

\medskip 

For \(q \in [1,+\infty)\), we denote by \(\lfloor q \rfloor\) the integer part of \(q\) and by \(\pi_{\lfloor q \rfloor}(N)\) the \( \lfloor q \rfloor^\textrm{th}\) homotopy group of \(N\). For instance, if \(\pi_{\lfloor q \rfloor}(N) \simeq \{0\}\), then every continuous map \(f \colon \mathbb{S}^{\lfloor q \rfloor} \to N\) on the \(\lfloor q \rfloor\)-dimensional sphere is homotopic to a constant map. 

\begin{propo}\label{convergenceHomotopy}
Let \(M,N\) be two smooth connected compact Riemannian manifolds and let \(p\in [1,+\infty)\) so that \(1 \le 2p < m\). 
\begin{enumerate}[(i)]
\item\label{simpleCase} If for every \(u \in \dot{W}^{2,p}(M, N)\), there exists a sequence \((u_\ell)_{\ell\in \N}\) in \(C^\infty(M, N)\) such that the sequence \((Tu_\ell)_{\ell \in \N}\) converges to \(Tu\) locally in measure and such that the sequence \((\abs{D^2_Ku_\ell}_{g^*_M\otimes g^*_M\otimes g_N})_{\ell \in \N}\) is bounded in \(L^p(M)\), and if \(2p \notin \N\), then \(\pi_{\lfloor 2p \rfloor}(N) \simeq \{0\}\). 
\item\label{criticalCase} If for every \(u \in \dot{W}^{2,p}(M, N)\), there exists a sequence \((u_\ell)_{\ell\in \N}\) in \(C^\infty(M, N)\) such that the sequence \((T^2u_\ell)_{\ell \in \N}\) converges to \(T^2u\) locally in measure and such that the sequence \((\abs{D^2_Ku_\ell}_{g^*_M\otimes g^*_M\otimes g_N})_{\ell \in \N}\) converges to \(\abs{D^2_K u}_{g^*_M\otimes g^*_M\otimes g_N}\) in \(L^p(M)\), and if \(2p \in \N\), then \(\pi_{2p}(N) \simeq \{0\}\). 
\end{enumerate}
\end{propo}

As a consequence, if \(m-1 < 2p < m\), since \(\pi_{m-1} (\mathbb{S}^{m-1}) \not\simeq \{0\}\), there is no weakly bounded sequence \((u_\ell)_{\ell \in \N}\) in \(\dot{W}^{2,p}(\mathbb{B}^m,\mathbb{S}^{m-1})\) such that \((Tu_\ell)_{\ell \in \N}\) converges to \(Tu_h\) locally in measure and if \(m-1 \le 2p < m\), the map \(u_h\) cannot be strongly approximated by smooth maps in  \(\dot{W}^{2,p}(\mathbb{B}^m,\mathbb{S}^{m-1})\). 

\begin{lemme}\label{lemmeConvergence}
Let \(N\) be a smooth connected compact Riemannian manifold and let \(p\in [1,k)\). 
Let \((u_\ell)_{\ell \in \N}\) be a sequence in \(C^\infty(\mathbb{S}^k,N)\) and let \(u \in \dot{W}^{2,p}(\mathbb{S}^k, N)\). 
\begin{enumerate}[(i)]
\item\label{lemmeSimple} 
If \(2p \ge k\), if the sequence \((Tu_\ell)_{\ell \in \N}\) converges to \(T u\) locally in measure and if the sequence \((\abs{D^2_K u_\ell}_{g^*_{k+1}\otimes g^*_{k+1} \otimes g_N})_{\ell \in \N}\) is bounded in \(L^p(\mathbb{S}^k)\), then the sequence \((\abs{Tu_\ell}_{g^*_{k+1} \otimes g_N})_{\ell \in \N}\) is bounded in \(L^{2p}(\mathbb{S}^k)\). 
\item\label{lemmeCritical} 
If \(2 p \ge k\), if the sequence \((T^2 u_\ell)_{\ell \in \N}\) converges to \(T^2 u\) locally in measure and if the sequence \((\abs{D^2_K u_\ell}_{g^*_{k+1}\otimes g^*_{k+1} \otimes g_N})_{\ell \in \N}\) converges to \(\abs{D^2_K u}_{g^*_{k+1} \otimes g^*_{k+1} \otimes g_N}\) in \(L^p(\mathbb{S}^k)\), 
then the sequence \((\abs{Tu_\ell}_{g^*_{k+1}\otimes g_N})_{\ell \in \N}\) converges to \(\abs{Tu}_{g^*_{k+1}\otimes g_N}\) in \(L^{2p}(\mathbb{S}^k)\). 
\end{enumerate}
\end{lemme}

Since Sobolev spaces between Riemannian manifolds do not form a vector space, it is not surprising that there is no some equivalence between both assertions in the previous lemma.  
\begin{lemme}[Poincar\'e inequalities]\label{Poincare}
Let \(M\) be a smooth connected compact Riemannian manifold and let \(\mu\) be the measure associated to the Riemannian metric on \(M\). 
Let \(p \in [1,m)\) and let \(q \in [1, \frac{mp}{m-p}]\). 
Let \(\varepsilon >0\). 
There exists \(C >0\) such that for every map \(v \in W^{1,p}(M)\) and every measurable subset \(A \subseteq M\) such that \(\mu(A) > \varepsilon\), 
\[
\norm{v}_{L^q(M)} \le C \Bigl(\norm{\nabla v}_{L^p(M)} + \frac{1}{\mu(A)}\int_A \abs{v} \Bigr). 
\]
\end{lemme}

\begin{proof}
\resetconstant
By the classical Poincar\'e inequality on compact Riemannian manifolds~\cite{hebey}*{proposition 3.9}, there exists a constant \(\Cl{poin}>0\) such that for every \(v\in W^{1,p}(M)\), 
\[
\norm{v}_{L^q(M)} \le \Cr{poin} \norm{\nabla v}_{L^p(M)} + \mu(M)^{\frac{1}{q}-1} \int_M \abs{v}.
\] 
By H\"older's inequality applied to the second term, for every measurable subset \(A \subseteq M\) such that \(\mu(A) >\varepsilon\), we have
\[
\norm{v}_{L^q(M)} \le \Cr{poin} \norm{\nabla v}_{L^p(M)} + \frac{\mu(M\setminus A)^{1-\frac{1}{q}}}{\mu(M)^{1-\frac{1}{q}}} \norm{v}_{L^q(M)} + \frac{\mu(M)^\frac{1}{q}}{\mu(A)} \int_A \abs{v},
\]
and so 
\begin{align*}
\norm{v}_{L^q(M)} 
& \le \frac{\mu(M)}{\mu(M)^{1-\frac{1}{q}} - \mu(M\setminus A)^{1-\frac{1}{q}}} 
\Bigl( \Cr{poin} \, \mu(M)^{-\frac{1}{q}} \norm{\nabla v}_{L^q(M)} + \frac{1}{\mu(A)} \int_A \abs{v} \Bigr) \\
& \le  \frac{\mu(M)}{\mu(M)^{1-\frac{1}{q}} - (\mu(M)-\varepsilon)^{1-\frac{1}{q}}} 
\Bigl(\Cr{poin} \, \mu(M)^{-\frac{1}{q}} \norm{\nabla v}_{L^q(M)} + \frac{1}{\mu(A)} \int_A \abs{v} \Bigr). \qedhere
\end{align*}
\end{proof}

\begin{proof}[Proof of lemma~\ref{lemmeConvergence}]
\resetconstant
Let \(\mu\) be the measure associated to the Riemannian metric on the sphere \(\mathbb{S}^k \subseteq \R^{k+1}\). 
Since the sequence \((T u_\ell)_{\ell \in \N}\) converges to \(Tu\) locally in measure in assertions~\eqref{lemmeSimple} and~\eqref{lemmeCritical}, there exist \(\varepsilon> 0\) and \(\gamma >0\) such that for every \(\ell \in \N\), 
\[
\mu \bigl(\{ x \in \mathbb{S}^k \colon \abs{Tu_\ell(x)}_{g^*_{k+1} \otimes g_N} \le \gamma\} \bigr) > \varepsilon. 
\]
For every \(\ell \in \N\), by lemma~\ref{absTu}, \(\abs{Tu_\ell}_{g^*_{k+1} \otimes g_N} \in W^{1,p}(\mathbb{S}^k)\) and almost everywhere in \(\mathbb{S}^k\)
\[
\bigabs{T\abs{Tu_\ell}_{g^*_{k+1}\otimes g_N}}_{g^*_{k+1} \otimes g_1} \le \abs{D^2_K u_\ell}_{g^*_{k+1} \otimes g^*_{k+1} \otimes g_N}. 
\]
So by the Poincar\'e inequality (lemma~\ref{Poincare}), there exists \(\Cl{alpha} >0\) such that for every \(\ell \in \N\), 
if \(A_\ell = \{ x \in \mathbb{S}^k \colon \abs{T u_\ell(x)}_{g^*_{k+1} \otimes g_N} \le \gamma\}\),
\begin{multline}\label{TuBounded}
\norm{ \, \abs{T u_\ell}_{g^*_{k+1} \otimes g_N} \,}_{L^{2p}(\mathbb{S}^k)}
\\ \le \Cr{alpha} \Bigl( \bignorm{ \, \abs{T\abs{Tu_\ell}_{g^*_{k+1} \otimes g_N}}_{g^*_{k+1} \otimes g_1} \,}_{L^p(\mathbb{S}^k)} 
+ \frac{1}{\mu(A_\ell)} \int_{A_\ell} \abs{T u_\ell}_{g^*_{k+1}\otimes g_N} \Bigr) \\
\\ \le \Cr{alpha} \Bigl( \bignorm{ \, \abs{D^2_K u_\ell}_{g^*_{k+1} \otimes g^*_{k+1} \otimes g_N} \, }_{L^p(\mathbb{S}^k)} + \gamma \Bigr).
\end{multline}
If assumptions of assertion~\eqref{lemmeSimple} or~\eqref{lemmeCritical} are satisfied, by previous inequality~\eqref{TuBounded}, the sequence \((\abs{Tu_\ell}_{g^*_{k+1}\otimes g_N})_{\ell \in \N}\) is bounded in \(L^{2p}(\mathbb{S}^k)\). 

We assume now that assumptions of assertion~\eqref{lemmeCritical} are satisfied. 
Since the sequence \((u_\ell)_{\ell\in \N}\) converges to \(u\) locally in measure, by weak closure property in Sobolev spaces~\cite{cvs}*{proposition 3.8}, \(u \in \dot{W}^{1,2p}(\mathbb{S}^k, N)\). In particular, \(\abs{Tu}_{g^*_{k+1} \otimes g_N} \in L^{2p}(\mathbb{S}^k)\). 
Since \(u \in \dot{W}^{1,p}(\mathbb{S}^k, N) \cap \dot{W}^{2,p}(\mathbb{S}^k,N)\), by lemma~\ref{absTu}, \(\abs{Tu}_{g^*_{k+1}\otimes g_N} \in W^{1,p}(\mathbb{S}^k)\) and so by Sobolev inequalities on compact Riemannian manifolds~\cite{hebey}*{theorem 3.5}, 
there exists \(\Cl{sob} >0\) such that for every \(\ell \in \N\), 
\[
\norm{\, \abs{T u_\ell}_{g^*_{k+1} \otimes g_N} - \abs{T u}_{g^*_{k+1} \otimes g_N}\,}_{L^{2p}(\mathbb{S}^k)} 
\le \Cr{sob} \norm{\, \abs{T u_\ell}_{g^*_{k+1} \otimes g_N} - \abs{T u}_{g^*_{k+1} \otimes g_N}\,}_{W^{1,p}(\mathbb{S}^k)}. 
\]
By inequality~\eqref{TuBounded}, the sequence \((\abs{Tu_\ell}_{g^*_{k+1}\otimes g_N})_{\ell \in \N}\) is bounded in \(L^{2p}(\mathbb{S}^k)\), and so by H\"older's inequality, bounded in \(L^p(\mathbb{S}^k)\). 
Since the sequence \((\abs{D^2_K u_\ell}_{g^*_{k+1}\otimes g^*_{k+1} \otimes g_N})_{\ell \in \N}\) is bounded in \(L^p(\mathbb{S}^k)\), and by lemma~\ref{absTu}, the sequence \((\abs{Tu_\ell}_{g^*_{k+1}\otimes g_N})_{\ell \in \N}\) is bounded in \(W^{1,p}(\mathbb{S}^k)\). 
By Rellich--Kondrashov embedding theorem for Sobolev maps on compact Riemannian manifolds~\cite{aubin}*{theorem 2.34 (a)}, and since the sequence \((\abs{Tu_\ell}_{g^*_{k+1}\otimes g_N})_{\ell \in \N}\) converges to \(\abs{Tu}_{g^*_{k+1}\otimes g_N}\) in measure,  the sequence \((\abs{Tu_\ell}_{g^*_{k+1}\otimes g_N})_{\ell \in \N}\) converges to the map \(\abs{Tu}_{g^*_{k+1}\otimes g_N}\) in \(L^p(\mathbb{S}^k)\). 
Since \((T^2 u_\ell)_{\ell \in \N}\) converges to \(T^2u\) locally in measure, the sequence \((T\abs{T u_\ell}_{g^*_{k+1} \otimes g_N})_{\ell\in \N}\) converges to \(T\abs{T u}_{g^*_{k+1} \otimes g_N}\) locally in measure. Furthermore, for every \(\ell \in \N\), 
\[
\bigabs{T(\abs{T u_\ell}_{g^*_{k+1} \otimes g_N} - \abs{T u}_{g^*_{k+1} \otimes g_N})}_{g^*_{k+1}\otimes g_1} 
\le \abs{D^2_K u_\ell}_{g^*_{k+1}\otimes g^*_{k+1}\otimes g_N} + \abs{D^2_K u}_{g^*_{k+1}\otimes g^*_{k+1}\otimes g_N}.
\] 
By Lebesgue's dominated convergence theorem~\cite{bogachev}*{theorem 2.8.5}, the sequence \\ 
\((T(\abs{T u_\ell}_{g^*_{k+1} \otimes g_N} - \abs{T u}_{g^*_{k+1} \otimes g_N}))_{\ell\in \N}\) converges to \(0\) in \(L^p(\mathbb{S}^k)\). 
Hence, the sequence \((\abs{Tu_\ell}_{g^*_{k+1}\otimes g_N})_{\ell \in \N}\) converges to \(\abs{Tu}_{g^*_{k+1}\otimes g_N}\) in \(W^{1,p}(\mathbb{S}^k)\), and so in \(L^{2p}(\mathbb{S}^{k})\) by the Sobolev inequality. 
\end{proof}

\begin{proof}[Proof of proposition~\ref{convergenceHomotopy}]
We first give the proof when \(M = \overline{\Bset}^{\floor{2 p} + 1} \times L\). 
If \(\pi_{\lfloor 2p \rfloor}(N) \not\simeq \{0\}\), then there exists a smooth map \(f \colon \mathbb{S}^{\lfloor 2p \rfloor} \to N\) which is not homotopic to a constant map in \(C^0(\mathbb{S}^{\lfloor 2p \rfloor}, N)\). 
We define the map \(u \colon M \to N\) for every \(x = (x',x'') \in \overline{\mathbb{B}}^{\lfloor 2p \rfloor +1} \times L\) by \(u(x) = f(\frac{x'}{\abs{x'}})\). 
We observe that there exists a constant \(C >0\) such that for every \(x \in (\overline{\Bset}^{\floor{2p} +1} \setminus \{0\}) \times L\), \(\abs{D^2_K u (x)}_{g^*_M\otimes g^*_M \otimes g_N} \le C/\abs{x'}^{2}\), and thus \(u\in \dot{W}^{2,p}(M,N)\). 

We first show that the assumptions of assertion~\eqref{simpleCase} cannot be satisfied. Otherwise, since for every \(\ell \in \N\), 
\begin{align*}
\int_{M} \abs{D^2_K u_\ell}_ {g^*_M \otimes g^*_M \otimes g_N}^p 
& = \int_{L} \int_{\mathbb{B}^{\lfloor 2p \rfloor +1}} \abs{D^2_K u_\ell}_ {g^*_M \otimes g^*_M \otimes g_N}^p \\
& = \int_{L} \Bigl( \int_0^1 \Bigl( \int_{\mathbb{S}^{\lfloor 2p \rfloor}_r} \abs{D^2_K u_\ell}_ {g^*_M \otimes g^*_M \otimes g_N}^p \Bigr) \, d r \Bigr), 
\end{align*}
by Fatou's lemma, for almost every \(x'' \in L\) and almost every \(r\in (0,1)\), up to a subsequence, the sequence \(((\abs{D^2_K u_\ell}_ {g^*_M \otimes g^*_M \otimes g_N})_{\arrowvert{\mathbb{S}^{\lfloor 2p \rfloor }_r \times \{x''\}}})_{\ell \in \N}\) is bounded in \(L^p(\mathbb{S}^{\lfloor 2p \rfloor}_r)\). 
By lemma~\ref{lemmeConvergence}~\eqref{lemmeSimple}, for almost every \(x'' \in L\) and almost every \(r\in (0,1)\), the sequence \(((\abs{Tu_\ell}_ {g^*_M \otimes g_N})_{\arrowvert{\mathbb{S}^{\lfloor 2p \rfloor }_r \times \{x''\}}})_{\ell \in \N}\) is bounded in \(L^{2p}(\mathbb{S}^{\lfloor 2p \rfloor}_r)\). 
Hence, for almost every \(x'' \in L\) and almost every \(r\in (0,1)\), the sequence \(((u_\ell)_{\arrowvert{\mathbb{S}^{\lfloor 2p \rfloor}_r \times \{x''\}}})_{\ell\in \N}\) is bounded in \(\dot{W}^{1,2p}(\mathbb{S}^{\lfloor 2p \rfloor}_r,N)\) and since the homotopy classes are preserved by weakly bounded sequence in \(\dot{W}^{1,2p}(\mathbb{S}_r^{\lfloor 2p \rfloor},\) \(N)\)~\cite{whi}*{theorem 2.1}, 
the map \(u_{\arrowvert{\mathbb{S}^{\lfloor 2p \rfloor}_r \times \{x''\}}}\) is  homotopic to a constant map, and so the map \(f\) also. 

We now show that the assumptions of assertion~\eqref{criticalCase} cannot be satisfied. Otherwise, since for every \(\ell \in \N\), 
\begin{multline*}
\int_{M} \abs{ \, \abs{D^2_K u_\ell}_ {g^*_M \otimes g^*_M \otimes g_N} - \abs{D^2_K u}_ {g^*_M \otimes g^*_M \otimes g_N} }^p  
\\ 
= \int_L \Bigl( \int_0^1 \Bigl( 
\int_{\mathbb{S}^{2p}_r} \abs{ \, \abs{D^2_K u_\ell}_ {g^*_M \otimes g^*_M \otimes g_N} - \abs{D^2_K u}_ {g^*_M \otimes g^*_M \otimes g_N}}^p \Bigr) \, dr \Bigr), 
\end{multline*}
up to a subsequence, for almost every \(x'' \in L\) and almost every \(r\in (0,1)\), the sequence 
\(((\abs{D^2_K u_\ell}_{g^*_M \otimes g^*_M \otimes g_N})_{\arrowvert{\mathbb{S}^{2p}_r \times \{x''\}}})_{\ell \in \N}\) converges to \((\abs{D^2_K u}_{g^*_M \otimes g^*_M \otimes g_N})_{\arrowvert{\mathbb{S}^{2p}_r \times \{x''\}}}\) in \(L^p(\mathbb{S}^{2p}_r)\). 
By lemma~\ref{lemmeConvergence}~\eqref{lemmeCritical}, for almost every \(x'' \in L\) and almost every \(r\in (0,1)\), the sequence \(((\abs{Tu_\ell}_{g^*_M \otimes g_N})_{\arrowvert{\mathbb{S}^{2p}_r \times \{x''\}}} )_{\ell \in \N}\) converges to \((\abs{Tu}_{g^*_M \otimes g_N})_{\arrowvert{\mathbb{S}^{2p}_r \times \{x''\}}}\) in \(L^{2p}(\mathbb{S}^{2p}_r)\). 
Hence, for almost every \(x'' \in L\) and almost every \(r\in (0,1)\), the sequence \(((u_\ell)_{\arrowvert{\mathbb{S}^{2p}_r \times \{x''\}}})_{\ell\in \N}\) converges to \(u_{\arrowvert{\mathbb{S}^{2p}_r \times \{x''\}}}\) in \(\dot{W}^{1,2p}(\mathbb{S}^{2p}_r, N)\)~\cite{cvs}*{proposition 4.4} and since the homotopy classes are preserved by strong convergence in \(\dot{W}^{1,2p}(\mathbb{S}_r^{2p},N)\)~\cite{white}, the map \(u_{\arrowvert{\mathbb{S}^{2p}_r \times \{x''\}}}\) is homotopic to a constant map, and so the map \(f\) also. 

%------
For a general manifold \(M\), we start from a counterexample \(u : \overline{\Bset}^{\floor{2 p} + 1} \times \mathbb{S}^{m - \floor{2 p} - 1} \to N\) to a Sobolev map on the whole of \(M\). 
Indeed, there is a map \(\Phi : \Bset^m \to \mathbb{R}^m\) which maps diffeomorphically some subset \(A \subset \mathbb{B}^m\) to \(\mathbb{B}^{\floor{2p} +1} \times \mathbb{S}^{m - \floor{2p} - 1}\) (\(A\) is a tubular neighborhood of an \((m - \floor{2p} - 1)\)--dimensional sphere, with \(\mathbb{S}^{0} = \{-1, 1\}\)), 
\(\Phi (\mathbb{B}^m \setminus A) \subset (\R^{\floor{2p} + 1} \setminus \{0\}) \times \R^{m - \floor{2p} - 1}\) 
and \(\Phi\) is constant near \(\partial \Bset^m\). The map \(u \circ  \Phi\) is constant near the boundary and can thus be transported and extended to any \(m\)--dimensional manifold, giving the required counterexample.
\end{proof}

For the notion of intrinsic higher-order Sobolev spaces, if the target manifold is compact, the space is also larger than the one by embedding (proposition~\ref{iotaCompact}). In this case, the condition \(\pi_{\lfloor k p \rfloor}(N) \simeq \{0\}\) which is necessary and sufficient in the definition by embedding~\cite{bpvs}*{theorem 1} is also necessary. 

\begin{propo}\label{convergenceHomotopyHigher}
Let \(M,N\) be two smooth connected compact Riemannian manifolds and let \(p\in [1,+\infty)\) so that \(1 \le kp < m\). 
\begin{enumerate}[(i)]
\item 
If for every \(u \in \dot{W}^{k,p}(M, N)\), there exists a sequence \((u_\ell)_{\ell\in \N}\) in \(C^\infty(M, N)\) such that the sequence \((T^{k-1} u_\ell)_{\ell \in \N}\) converges to \(T^{k-1}u\) locally in measure and such that the sequence \((\abs{D^k_K u_\ell}_{(\otimes^k g^*_M) \otimes g_N})_{\ell \in \N}\) is bounded in \(L^p(M)\), and if \(k p \notin \N\), then \(\pi_{\lfloor kp \rfloor}(N) \simeq \{0\}\). 
\item
If for every \(u \in \dot{W}^{k,p}(M, N)\), there exists a sequence \((u_\ell)_{\ell\in \N}\) in \(C^\infty(M, N)\) such that the sequence \((T^k u_\ell)_{\ell \in \N}\) converges to \(T^k u\) locally in measure and such that the sequence \((\abs{D^k_K u_\ell}_{(\otimes^k g^*_M) \otimes g_N})_{\ell \in \N}\) converges to \(\abs{D^k_K u}_{(\otimes^k g^*_M) \otimes g_N}\) in \(L^p(M)\), and if \(kp \in \N\), then \(\pi_{k p}(N) \simeq \{0\}\). 
\end{enumerate}
\end{propo}

To prove this proposition, we can also state a similar lemma to lemma~\ref{lemmeConvergence}. 
\begin{lemme}
Let \(N\) be a smooth connected compact Riemannian manifold and let \(p\in [1,\infty)\) so that \((k-1)p < q\). 
Let \((u_\ell)_{\ell \in \N}\) be a sequence in \(C^\infty(\mathbb{S}^q,N)\) and let \(u \in \dot{W}^{k,p}(\mathbb{S}^q, N)\). 
\begin{enumerate}[(i)]
\item If \(kp \ge q\), if the sequence \((T^{k-1} u_\ell)_{\ell \in \N}\) converges to \(T^{k-1} u\) locally in measure and if the sequence \((\abs{D^k_K u_\ell}_{(\otimes^k g^*_{q+1}) \otimes g_N})_{\ell \in \N}\) is bounded in \(L^p(\mathbb{S}^q)\), then the sequence \((\abs{Tu_\ell}_{g^*_{q+1}\otimes g_N})_{\ell \in \N}\) is bounded in \(L^{kp}(\mathbb{S}^q)\). 
\item If \(kp \ge q\), if the sequence \((T^k u_\ell)_{\ell \in \N}\) converges to \(T^k u\) locally in measure and if the sequence \((\abs{D^k_K u_\ell}_{(\otimes^k g^*_{q+1}) \otimes g_N})_{\ell \in \N}\) converges to \(\abs{D^k_K u}_{(\otimes^k g^*_{q+1}) \otimes g_N}\) in \(L^p(\mathbb{S}^q)\), then the sequence \((\abs{Tu_\ell}_{g^*_{q+1}\otimes g_N})_{\ell \in \N}\) converges to \(\abs{Tu}_{g^*_{q+1}\otimes g_N}\) in \(L^{kp}(\mathbb{S}^q)\). 
\end{enumerate}
\end{lemme}

To prove such a lemma, we can use Poincar\'e inequalities (lemma~\ref{Poincare}) recursively and lemma~\ref{absDku}. The structure of the proof of proposition~\ref{convergenceHomotopyHigher} is then the same as the one of proposition~\ref{convergenceHomotopy}. 

%********************************************************


\begin{bibdiv}
\begin{biblist}
\bib{adams}{book}{
   author={Adams, David R.},
   author={Hedberg, Lars Inge},
   title={Function spaces and potential theory},
   series={Grundlehren der Mathematischen Wissenschaften},
   volume={314},
   publisher={Springer, Berlin},
   date={1996},
   pages={xii+366},
}
\bib{angels}{article}{
   author={Angelsberg, Gilles},
   author={Pumberger, David},
   title={A regularity result for polyharmonic maps with higher integrability},
   journal={Ann. Global Anal. Geom.},
   volume={35},
   date={2009},
   number={1},
   pages={63--81},
}
\bib{aubin}{book}{
   author={Aubin, Thierry},
   title={Nonlinear analysis on manifolds},
   subtitle={Monge--Amp\`ere equations},
   series={Grundlehren der Mathematischen Wissenschaften},
   volume={252},
   publisher={Springer}, 
   address={New York},
   date={1982},
   pages={xii+204},
}
\bib{bertram}{article}{
   author={Bertram, Wolfgang},
   title={Differential geometry over general base fields and rings},
   conference={
      title={Modern trends in geometry and topology},
   },
   book={
      publisher={Cluj Univ. Press}, 
      address={Cluj-Napoca},
   },
   date={2006},
   pages={95--101},
}
\bib{bethuel}{article}{
   author={Bethuel, Fabrice},
   title={The approximation problem for Sobolev maps between two manifolds},
   journal={Acta Math.},
   volume={167},
   date={1991},
   number={3-4},
   pages={153--206},
}
\bib{bz}{article}{
   author={Bethuel, Fabrice},
   author={Zheng, Xiao Min},
   title={Density of smooth functions between two manifolds in Sobolev spaces},
   journal={J. Funct. Anal.},
   volume={80},
   date={1988},
   number={1},
   pages={60--75},
}
\bib{bogachev}{book}{
   author={Bogachev, V. I.},
   title={Measure theory},
   publisher={Springer},
   place={Berlin},
   date={2007},
   pages={Vol. I: xviii+500 pp., Vol. II: xiv+575},
}
\bib{bourdaud}{article}{
  author={Bourdaud, G.},
  title={Le calcul fonctionnel dans les espaces de Sobolev},
  conference={
    title={S\'eminaire sur les \'Equations aux D\'eriv\'ees Partielles},
    date={1990--1991},
  },
  book={
    publisher={\'Ecole Polytech.}, 
    address={Palaiseau, France},
  },
  date={1991},
}
\bib{bpvs}{article}{
   author={Bousquet, Pierre},
   author={Ponce, Augusto C.},
   author={Van Schaftingen, Jean},
   title={Strong density for higher order Sobolev spaces into compact manifolds},
   journal={J. Eur. Math. Soc. (JEMS)},
   volume={17},
   date={2015},
   number={4},
   pages={763--817},
}
\bib{brezis}{book}{
   author={Brezis, Haim},
   title={Functional analysis, Sobolev spaces and partial differential
   equations},
   series={Universitext},
   publisher={Springer},
   place={New York},
   date={2011},
   pages={xiv+599},
}
\bib{chang}{article}{
   author={Chang, Sun-Yung A.},
   author={Wang, Lihe},
   author={Yang, Paul C.},
   title={A regularity theory of biharmonic maps},
   journal={Comm. Pure Appl. Math.},
   volume={52},
   date={1999},
   number={9},
   pages={1113--1137},
}
\bib{cvs}{article}{
  author={Convent, Alexandra},
  author={Van Schaftingen, Jean},
  title={Intrinsic colocal weak derivatives and Sobolev spaces between manifolds},
  journal={Ann. Sc. Norm. Super. Pisa Cl. Sci. (5).},
  volume={16},
  date={2016},
  number={1},
  pages={97--128},
}
\bib{cvs2}{article}{
   author={Convent, Alexandra},
   author={Van Schaftingen, Jean},
   title={Geometric partial differentiability on manifolds: the tangential derivative and the chain rule},
   journal={J. Math. Anal. Appl.},
   volume={435},
   date={2016},
   number={2},
   pages={1672--1681},
}
\bib{dahlberg}{article}{
   author={Dahlberg, Bj{\"o}rn E. J.},
   title={A note on Sobolev spaces},
   conference={
      title={Harmonic analysis in Euclidean spaces},
      address={Williamstown, Mass.},
      date={1978},
   },
   book={
      series={Proc. Sympos. Pure Math., XXXV, Part},
      publisher={Amer. Math. Soc., Providence, R.I.},
   },
   date={1979},
   pages={183--185},
}
\bib{derham}{book}{
   author={de Rham, Georges},
   title={Differentiable manifolds},
   series={Grundlehren der Mathematischen Wissenschaften },
   volume={266},
   subtitle={Forms, currents, harmonic forms},
   publisher={Springer},
   place={Berlin},
   date={1984},
   pages={x+167},
   isbn={3-540-13463-8},
}
\bib{dieudonne}{book}{
   author={Dieudonn{\'e}, J.},
   title={\'El\'ements d'analyse},
   volume={III},
   series={Cahiers Scientifiques, Fasc. XXXIII},
   publisher={Gauthier-Villars}, 
   address={Paris},
   date={1970},
   pages={xix+367},
}
\bib{docarmo}{book}{
   author = {do Carmo, Manfredo},
   title = {Riemannian geometry},
   translator = {Flaherty, Francis},
   publisher = {Birkha\"user},
   address = {Boston, Mass.},
   date = {1992},
   series = {Mathematics: Theory and Applications},
}
\bib{eg}{book}{
   author = {Evans, Lawrence C.},
   author = {Gariepy, Ronald F.},
   title = {Measure theory and fine properties of functions},
   date = {1992},
   publisher = {CRC Press},
   address = {Boca Raton, Fla.},
}
\bib{gagliardo}{article}{
   author={Gagliardo, Emilio},
   title={Ulteriori propriet\`a di alcune classi di funzioni in pi\`u variabili},
   journal={Ricerche Mat.},
   volume={8},
   date={1959},
   pages={24--51},
}
\bib{gastel}{article}{
   author={Gastel, Andreas},
   author={Scheven, Christoph},
   title={Regularity of polyharmonic maps in the critical dimension},
   journal={Comm. Anal. Geom.},
   volume={17},
   date={2009},
   number={2},
   pages={185--226},
}
\bib{gold}{article}{
   author={Goldstein, Pawe{\l}},
   author={Strzelecki, Pawe{\l}},
   author={Zatorska-Goldstein, Anna},
   title={On polyharmonic maps into spheres in the critical dimension},
   journal={Ann. Inst. H. Poincar\'e Anal. Non Lin\'eaire},
   volume={26},
   date={2009},
   number={4},
   pages={1387--1405},
}
\bib{gong}{article}{
   author={Gong, Huajun},
   author={Lamm, Tobias},
   author={Wang, Changyou},
   title={Boundary partial regularity for a class of biharmonic maps},
   journal={Calc. Var. Partial Differential Equations},
   volume={45},
   date={2012},
   number={1--2},
   pages={165--191},
}
\bib{gr}{article}{
   author={Grabowski, Janusz},
   author={Rotkiewicz, Miko{\l}aj},
   title={Higher vector bundles and multi-graded symplectic manifolds},
   journal={J. Geom. Phys.},
   volume={59},
   date={2009},
   number={9},
   pages={1285--1305},
}
\bib{gm}{article}{
   author={Gracia-Saz, Alfonso},
   author={Mackenzie, Kirill Charles Howard},
   title={Duality functors for triple vector bundles},
   journal={Lett. Math. Phys.},
   volume={90},
   date={2009},
   number={1--3},
   pages={175--200},
}
\bib{gk}{article}{
   author={Gudmundsson, Sigmundur},
   author={Kappos, Elias},
   title={On the geometry of tangent bundles},
   journal={Expo. Math.},
   volume={20},
   date={2002},
   number={1},
   pages={1--41},
}
\bib{hajlasz}{article}{
   author={Haj{\l}asz, Piotr},
   title={Sobolev mappings between manifolds and metric spaces},
   conference={
      title={Sobolev spaces in mathematics. I},
   },
   book={
      series={Int. Math. Ser. (N. Y.)},
      volume={8},
      publisher={Springer}, 
      address={New York},
   },
   date={2009},
   pages={185--222},
}
\bib{ht}{article}{
   author={Haj{\l}asz, Piotr},
   author={Tyson, Jeremy T.},
   title={Sobolev Peano cubes},
   journal={Michigan Math. J.},
   volume={56},
   date={2008},
   number={3},
   pages={687--702},
}
\bib{hl}{article}{
   author={Hang, Fengbo},
   author={Lin, Fanghua},
   title={Topology of Sobolev mappings. II},
   journal={Acta Math.},
   volume={191},
   date={2003},
   number={1},
   pages={55--107},
}
\bib{hebey}{book}{
   author={Hebey, Emmanuel},
   title={Sobolev spaces on Riemannian manifolds},
   series={Lecture Notes in Mathematics},
   volume={1635},
   publisher={Springer}, 
   address={Berlin},
   date={1996},
   pages={x+116},
}
\bib{hirsch}{book}{
   author={Hirsch, Morris W.},
   title={Differential topology},
   series={Graduate Texts in Mathematics}, 
   volume={33},
   publisher={Springer},
   place={New York},
   date={1976},
   pages={x+221},
}
\bib{hormander}{book}{
   author={H{\"o}rmander, Lars},
   title={The analysis of linear partial differential operators},
   part = {I},
   edition={2},
   subtitle={Distribution theory and Fourier analysis},
   publisher={Springer},
   place={Berlin},
   date={1990},
}
\bib{hm}{article}{
   author={Hornung, Peter},
   author={Moser, Roger},
   title={Intrinsically $p$-biharmonic maps},
   journal={Calc. Var. Partial Differential Equations},
   volume={51},
   date={2014},
   number={3-4},
   pages={597--620},
}
\bib{klingenberg}{book}{
   author={Klingenberg, Wilhelm},
   title={Lectures on closed geodesics},
   series={Grundlehren der Mathematischen Wissenschaften}, 
   volume={230},
   publisher={Springer}, 
   address={Berlin--New York},
   date={1978},
   pages={x+227},
}
\bib{lamm}{article}{
   author={Lamm, Tobias},
   author={Wang, Changyou},
   title={Boundary regularity for polyharmonic maps in the critical dimension},
   journal={Adv. Calc. Var.},
   volume={2},
   date={2009},
   number={1},
   pages={1--16},
}
\bib{lee}{book}{
   author={Lee, John M.},
   title={Introduction to topological manifolds},
   series={Graduate Texts in Mathematics},
   volume={202},
   edition={2},
   publisher={Springer}, 
   address={New York},
   date={2011},
   pages={xviii+433},
}
\bib{mackenzie}{article}{
   author={Mackenzie, Kirill C. H.},
   title={Duality and triple structures},
   conference={
      title={The breadth of symplectic and Poisson geometry},
   },
   book={
      series={Progr. Math.},
      volume={232},
      publisher={Birkh\"auser},
      address={Boston, Mass.},
   },
   date={2005},
   pages={455--481},
}
\bib{michor}{book}{
   author={Michor, Peter W.},
   title={Topics in differential geometry},
   series={Graduate Studies in Mathematics},
   volume={93},
   publisher={American Mathematical Society}, 
   address={Providence, R.I.},
   date={2008},
   pages={xii+494},
}
\bib{moser}{article}{
   author={Moser, Roger},
   title={A variational problem pertaining to biharmonic maps},
   journal={Comm. Partial Differential Equations},
   volume={33},
   date={2008},
   number={7-9},
   pages={1654--1689},
}
\bib{nash54}{article}{
   author={Nash, John},
   title={$C^1$ isometric imbeddings},
   journal={Ann. of Math. (2)},
   volume={60},
   date={1954},
   pages={383--396},
   issn={0003-486X},
}
\bib{nash56}{article}{
   author={Nash, John},
   title={The imbedding problem for Riemannian manifolds},
   journal={Ann. of Math. (2)},
   volume={63},
   date={1956},
   pages={20--63},
   issn={0003-486X},
}
\bib{nirenberg}{article}{
   author={Nirenberg, L.},
   title={On elliptic partial differential equations},
   journal={Ann. Scuola Norm. Sup. Pisa (3)},
   volume={13},
   date={1959},
   pages={115--162},
}
\bib{sasaki}{article}{
   author={Sasaki, Shigeo},
   title={On the differential geometry of tangent bundles of Riemannian manifolds. II},
   journal={T\^ohoku Math. J. (2)},
   volume={14},
   date={1962},
   pages={146--155},
}
\bib{su}{article}{
   author={Schoen, Richard},
   author={Uhlenbeck, Karen},
   title={Boundary regularity and the Dirichlet problem for harmonic maps},
   journal={J. Differential Geom.},
   volume={18},
   date={1983},
   number={2},
   pages={253--268},
}
\bib{struwe}{article}{
   author={Struwe, Michael},
   title={Partial regularity for biharmonic maps, revisited},
   journal={Calc. Var. Partial Differential Equations},
   volume={33},
   date={2008},
   number={2},
   pages={249--262},
}
\bib{sw}{article}{
   author={Sun, XiaoWei},
   author={Wang, YouDe},
   title={New geometric flows on Riemannian manifolds and applications to Schr\"odinger--Airy flows},
   journal={Sci. China Math.},
   volume={57},
   date={2014},
   number={11},
   pages={2247--2272},
}
\bib{urakawa}{article}{
   author={Urakawa, H.},
   title={The geometry of biharmonic maps},
   conference={
      title={Harmonic maps and differential geometry},
   },
   book={
      series={Contemp. Math.},
      volume={542},
      publisher={Amer. Math. Soc.}, 
      address={Providence, R.I.},
   },
   date={2011},
   pages={159--175},
}
\bib{wendl}{book}{
   author={Wendl, C.},
   title={Lectures Notes on Bundles and Connections},
   date={2008},
}
\bib{white}{article}{
   author={White, Brian},
   title={Infima of energy functionals in homotopy classes of mappings},
   journal={J. Differential Geom.},
   volume={23},
   date={1986},
   number={2},
   pages={127--142},
}
\bib{whi}{article}{
   author={White, Brian},
   title={Homotopy classes in Sobolev spaces and the existence of energy minimizing maps},
   journal={Acta Math.},
   volume={160},
   date={1988},
   number={1-2},
   pages={1--17},
}
\bib{willem}{book}{
  author = {Willem, Michel},
  title = {Functional analysis},
  subtitle = {Fundamentals and Applications},
  series={Cornerstones},
  publisher = {Birkh\"auser},
  place = {Basel},
  volume = {XIV},
  pages = {213},
  date={2013},
}
\end{biblist}
\end{bibdiv}
\end{document}